\documentclass[11pt]{amsart}
\usepackage{a4wide,xcolor,eucal,enumerate,mathrsfs}
\usepackage[normalem]{ulem}
\usepackage{amsfonts,amsmath,amssymb,amsthm}
\usepackage{pgf,tikz,color}
\usetikzlibrary{arrows}
\tikzstyle{dotnode} = [draw, fill, inner sep = 0pt, minimum size = 1.1mm, circle]
\tikzstyle{termnode} = [draw, fill = white, inner sep = 0pt, minimum size = 1.5mm, circle]

\definecolor{ffqqqq}{rgb}{1.,0.,0.}
\definecolor{uuuuuu}{rgb}{0.26666666666666666,0.26666666666666666,0.26666666666666666}
\definecolor{qqwuqq}{rgb}{0.,0.39215686274509803,0.}
\definecolor{qqqqff}{rgb}{0.,0.,1.}

\definecolor{zzttqq}{rgb}{0.6,0.2,0.}
\definecolor{qqqqff}{rgb}{0.,0.,1.}
\definecolor{ffvvqq}{rgb}{1.,0.3333333333333333,0.}
\definecolor{xdxdff}{rgb}{0.49019607843137253,0.49019607843137253,1.}
\newcommand\N{\mathbb{N}}
\newcommand\R{\mathbb{R}}

\def\P{\mathbb{P}}
\def\T{\mathbb{T}}
\def\B{\mathbb{B}}
\def\D{\mathbb{D}}

\def\diam{\mathrm{diam}\,}

\usepackage{color}

\newtheorem{theorem}{Theorem}[section]
\newtheorem{definition}[theorem]{Definition}

\newtheorem{lemma}[theorem]{Lemma}
\newtheorem{proposition}[theorem]{Proposition}
\newtheorem{corollary}[theorem]{Corollary}

\theoremstyle{remark}
\newtheorem{remark}[theorem]{Remark}
\newtheorem{example}[theorem]{Example}

\numberwithin{equation}{section}
\numberwithin{figure}{section}

\begin{document}

\title{Self-contracted curves have finite length}

\author{Eugene Stepanov}%
\address{
St.Petersburg Branch
of the Steklov Mathematical Institute of the Russian Academy of Sciences,
Fontanka 27,
191023 St.Petersburg,
Russia
\and
Department of Mathematical Physics, Faculty of Mathematics and Mechanics,
St. Petersburg State University, Universitetskij pr.~28, Old Peterhof,
198504 St.Petersburg, Russia%, email: stepanov.eugene@gmail.com
\and ITMO University
%\and
%%\address{
%Dipartimento di Matematica ``L.~Tonelli'', Universit\`{a} di
%Pisa, Largo B.~Pontecorvo~5, 56127 Pisa, Italy.
}
%\thanks{The work of the
%author was partially supported by the Italian government
%program ``Incentivazione alla mobilit\`{a} di studiosi stranieri e
%italiani residenti all'estero'' }
\email{stepanov.eugene@gmail.com}
\thanks{The work of the first author has been partially financed
by
RFBR grant \#17-01-00678 and by the Russian government grant \#074-U01, the Ministry of Education and Science of Russian Federation project \#14.Z50.31.0031. The second author
also acknowledges the support of the St.~Petersburg State University grant \#6.37.208.2016 as well as of the social investment program ``Native towns'',
of the PJSC ``Gazprom Neft''.
}

\author{Yana Teplitskaya}%
\address{
Chebyshev Laboratory, St. Petersburg State University, 14th Line V.O., 29B, 199178 St. Petersburg, Russia}

\begin{abstract}
A curve $\theta$: $I\to E$ in a metric space $E$ equipped with the distance $d$,
where $I\subset \R$ is a (possibly unbounded) interval,
is called self-contracted, if for any triple of instances of time
$\{t_i\}_{i=1}^3\subset I$ with $t_1\leq t_2\leq t_3$ one has $d(\theta(t_3),\theta(t_2))\leq d(\theta(t_3),\theta(t_1))$.
We prove that if $E$ is a finite-dimensional normed space with an arbitrary %, not necessarily symmetric
norm, the trace of
$\theta$ %is continuous and its trace
is bounded,
then $\theta$ has finite length, i.e.\ is rectifiable, thus answering positively the question raised in~\cite{Lemenant16sc-rectif}.
\end{abstract}

\subjclass{Primary 53A04, 52A21; Secondary 37N40, 49J52, 49J53, 65K10}

\maketitle

\section{Introduction}

Let $E$ be a metric space equipped with the distance $d$.
A curve $\theta$: $I\to E$, where $I\subset \R$ is an (possibly unbounded) interval,
is called self-contracted, if for any triple of instances of time
$\{t_i\}_{i=1}^3\subset I$ with $t_1\leq t_2\leq t_3$ one has $d(\theta(t_3),\theta(t_2))\leq d(\theta(t_3),\theta(t_1))$.
Of particular interest are continuous self-contracted curves in a finite-dimensional space $\R^n$ equipped with some norm.
In~\cite{Daniilid2010sc-asympt} and~\cite{Longinet15sc-convexgrad}
it has been shown that such curves arise as steepest decent curves for convex and level set convex (sometimes called also
quasi-convex) functions in the Euclidean space.
In~\cite{Daniilid2010sc-asympt} it has been proven that every %continuous
self-contracted curve in a bounded subset of $\R^2$ (equipped with the usual Euclidean distance)
necessarily has finite length, i.e.\ is \emph{rectifiable}. This result has been further extended to $\R^n$ with arbitrary $n\geq 1$, again equipped with
Euclidean norm, in~\cite{Daniilid15rectif} (and, independently, in~\cite{Longinet15sc-convexgrad} for continuous self-contracted curves) and to an arbitrary finite-dimensional Riemannian manifold in~\cite{Daniilid14sc-riem}.
Note that the self-contracted property
is quite sensible to the change
of the distance and even of the norm in $\R^n$, namely,
it is easy to provide examples of curves
self-contracted with respect to some norm and not self-contracted with respect to an equivalent one: for instance, a curve moving along three consecutive
sides of the square $[0,1]^2$ (say, clockwise, from the origin to
$(0,1)$, then to $(1,1)$ and finally to $(1,0)$) is self-contracted
with respect to the maximum (i.e.\ $\ell_2^\infty$) norm in $\R^2$,
but not with respect to the Euclidean one.
This raises the natural question whether it can be extended to self-contracted curves in $\R^n$ equipped with an arbitrary norm. This question has been posed in~\cite{Lemenant16sc-rectif}, and in the same paper a partial answer for uniformly convex smooth ($C^2$) norms has been given for the case $n=2$.
Here we give a positive answer for the case of a generic norm in $\R^n$, $n\geq 1$, not necessarily smooth. %or symmetric.

As opposed to~\cite{Daniilid2010sc-asympt,Daniilid15rectif} (and to~\cite{Lemenant16sc-rectif} which substantially extends the technique of~\cite{Daniilid15rectif}), where the proofs are based on finite ``continuous'' analysis arguments fundamentally relying on~\cite{Manselli91steepdec}, our technique has some ``discrete'' flavor. Namely, we first provide an estimate on self-contracted polygonal lines identified by the ordered
set $(A_1,\ldots,A_r)\in E^r$ of their vertices (of course, the use of the term ``polygonal line'' outside of a context of a linear vector space is an abuse of the language).
Namely, a vector of $(A_1,\ldots,A_r)\in E^r$
will be called \emph{self-contracted} with respect to the distance $d$ (or with respect to the norm $\|\cdot\|$, if $d$ is
coming from some norm $\|\cdot\|$), if $d(A_k, A_j)\leq d(A_k, A_i)$ whenever
$i\leq j\leq k$.
%, and \emph{strictly self-contracted}, if the above inequality is strict whenever $i\neq j$. %%%% ok ma non serve
We show that if $E$ is a finite-dimensional normed space, then the total Euclidean length of the self-contracted polygonal line is estimated by
the Euclidean distance between its first and last vertex. This immediately proves the desired result for self-contracted curves.
The proof of the estimate of the length of  self-contracted polygonal lines is more or less discrete in nature and uses a specially tailored induction together with some combinatorial type arguments on particular arrangements of sets of vertices.

\section{Notation and preliminaries}

For $\{A,B\}\subset \R^n$ we denote by $(AB)$ the unique line defined by these points (of course, when they are distinct), by $[AB]$ the closed line segment with endpoints $A$ and $B$, by $|AB|$ its Euclidean length. The notation $|\cdot|$ will always stand for the Euclidean norm in $\R^n$, and $a\cdot b$ will stand for the standard scalar product between $a\in \R^n$ and $b\in \R^n$. For two vectors $\{\nu_1,\nu_2\}\subset \R^n$ we denote by $\widehat{(\nu_1,\nu_2)}$ the angle between them, so that
$\widehat{(\nu_1,\nu_2)}\in [0,\pi]$.
If $\ell\subset\R^n$ is a line and $\Pi\subset\R^n$ is a linear subspace of arbitrary positive dimension,
we denote by $\widehat{(\ell,\Pi)}$ the angle between them (i.e.\ the minimum angle between vectors belonging to
$\ell$ and $\Pi$ respectively), so that $\widehat{(\ell,\Pi)}\in [0,\pi/2]$.
The angle at vertex $B$ of a triangle with vertices $A$, $B$, $C$ will be denoted by $\angle ABC$.
 For a $\D\subset \R^n$ we denote %by $\bar \D$ its closure,
by $\partial \D$ its topological boundary and by $\diam \D$ its Euclidean diameter (i.e.\ the diameter with respect to the Euclidean distance). For $\{a,b\}\in \R$ we write $a\vee b:=\max\{a,b\}$.

The notation $\nu^\bot$ for a $\nu\in \R^n$, unless otherwise explicitly defined, will stand for the linear subspace
$\{v\in \R^n\colon v\cdot \nu =0\}$, and $\Pi^\bot$ for a linear subspace $\Pi\subset \R^n$ will stand for its orthogonal
complement.

Fixed an $\varepsilon>0$, we call a segment $[AB]$ \emph{$\varepsilon$-horizontal} with respect to the linear subspace $\Pi\subset\R^n$ of arbitrary positive dimension,
if $\widehat{((AB), \Pi)}\leq \varepsilon$, and \emph{$\varepsilon$-vertical} with respect to this subspace otherwise
(we will abbreviate both notions to just \emph{horizontal} or \emph{vertical} respectively, if both the subspace and $\varepsilon$ are clear from the context). We use this notion in particular when the subspace $\Pi$ is one-dimensional
and coincides with some axis $x^j$ of some chosen coordinate system (the axis is seen as just a line, i.e.\ ignoring its direction). By $p_\Pi$ we denote the orthogonal projection onto $\Pi$.

If $E$ is an arbitrary set, and $(A_1, \ldots, A_r)\in E^r$ is an arbitrary vector of points of $E$, then
a vector $(A_{j_1},\ldots,A_{j_k})\in E^k$ with
$1\leq j_1 < j_2< \ldots < j_k\leq r$,
 will be called subvector
of $(A_1,\ldots,A_r)$, denoted by $(A_{j_1},\ldots,A_{j_k})\subset (A_1,\ldots,A_r)$.
If necessary, we identify the vector $(A_1, \ldots, A_r)\in E^r$ with the set $\{A_1, \ldots, A_r\}$, so that we write just
$(A_1, \ldots, A_r)\subset E$.
%Let $A_i\in \R^n$, $i=1,\ldots, r$, be arbitrary points.
In case $E=\R^n$, we
call the \emph{variation} of $(A_1, \ldots, A_r)\in E^r$ (denoted by $\ell(A_1,\ldots, A_r)$) the Euclidean length  of the
the polygonal line
$A_1\ldots A_r:=\cup_{j=1}^{r-1} [A_jA_{j+1}]$
i.e.\
\[
\ell(A_1,\ldots,A_r):= \sum_{i=1}^{r-1} |A_iA_{i+1}|.
\]
Further, for a linear subspace $\Pi\subset \R^n$ of arbitrary positive dimension we define
the variation of $(A_1,\ldots,A_r)$ along $\Pi$ by
\[
\ell_\Pi (A_1,\ldots,A_r) := \ell(p_\Pi(A_1),\ldots,p_\Pi(A_r)).
\]
For a curve $\theta\colon I\to \R^n$, where $I\subset \R$ is an interval (not necessarily finite), we denote by $\ell(\theta)$ its parametric length defined by the usual formula
\[
\ell(\theta):=\sup\left\{\ell\left(\{\theta(t_j)\}_{j=1}^m\right)\colon \{ t_j\}_{j=1}^m\subset I, m\in \N\right\}.
\]
%If $\ell(\theta)<+\infty$, then $\theta$ is usually called \emph{rectifiable}.

\section{Main results}

The following theorem is the first main result of this paper.

\begin{theorem}\label{th_scontr_estpolygonal1}
Let $A_i\in \R^n$, $i=1,\ldots, r$, and the vector $(A_1,\ldots,A_r)$ is self-contracted with respect to the norm $\|\cdot\|$,
then
\[
\ell(A_1,\ldots,A_r)\leq C|A_1A_r|
\]
for some $C>0$ depending only on $\|\cdot\|$  and on the space dimension $n$.
\end{theorem}

An immediate consequence of the above result is the following theorem on self-contracted curves giving the complete answer to the question posed in the Introduction.

\begin{theorem}\label{th_scontr_estcurve1}
Let $E$ be a finite-dimensional space equipped with the norm $\|\cdot\|$, and
let $\theta\colon I\to E$, where $I\subset \R$ is a (possibly unbounded) interval, be a %continuous
self-contracted curve with trace in a bounded set $\D\subset E$.
Then $\ell(\theta)\leq C\diam  \D$ for some $C>0$ depending only on $\|\cdot\|$ and on the space dimension $n$.
\end{theorem}

\begin{proof}
Consider an arbitrary finite set $\{t_i\}_{i=1}^r\subset I$, $t_i\leq t_{i+1}$. We have now that
for $A_i:=\theta(t_i)$ the vector $(A_1,\ldots, A_r)$ is self-contracted in $E$.
Thus
\[
\ell(A_1,\ldots,A_r)\leq C|A_1A_r|\leq C\diam  \D
\]
by Theorem~\ref{th_scontr_estpolygonal1}, concluding the proof.
%By Lemma~\ref{lm_sc1endpt1} there is
%an $y=\lim_{t\to +\infty} \theta(t)$. Without loss of generality we may assume that $y$ is the origin.
%Further, up to adding an additional point, we may assume then that $A_r=0$.
\end{proof}

\begin{remark}
The proofs of the above Theorems~\ref{th_scontr_estpolygonal1} and~\ref{th_scontr_estcurve1} never use essentially
the symmetry of the norm
$\|\cdot\|$. If the norm is not symmetric (i.e.\ not necessarily satisfying the assumption $\|-x\|=\|x\|$ for all $x\in E$), with the distance (now not necessarily symmetric anymore)
defined still by $d(x,y):=\|y-x\|$, then the geometric meaning of the self-contracted property of the curve
does not change with respect to the standard situation of a symmetric norm, that is, for every triple of instances of time
$\{t_i\}_{i=1}^3\subset I$ with $t_1\leq t_2\leq t_3$ and $\theta(t_1)$ on the boundary of a ball of the norm
(now a generic, not necessarily symmetric bounded convex absorbing set) centered at $\theta(t_3)$ one has
that $\theta(t_2)$ cannot be strictly outside of the latter ball.
Then using Lemma~\ref{lm_scontr_mediatr1} one has to change its claim as described
in Remark~\ref{rm_sc_mediatr1}. This would change the constant $3/4$ in the claim of Lemma~\ref{lm_scontr_horiz1} and hence
also the explicit constants in all the subsequent lemmata have to be substituted by constants dependent only on the norm $\|\cdot\|$, but all the respective results will remain true up to such modification of constants.
Thus both Theorem~\ref{th_scontr_estpolygonal1} and Theorem~\ref{th_scontr_estcurve1} are in fact true for possibly not symmetric norms.
\end{remark}

\begin{remark}
It is important to note that not every self-contracted curve with bounded trace in a finite-dimensional normed space  is continuous. In fact, the easy example $\theta\colon [0,1]\to \R$ defined by
$\theta(t):=0$ for $t\in [0,1/2)$ and $\theta(t):=1$ for $t\in [1/2,1]$ provides a discontinuous self-contracted curve even in $\R$.
\end{remark}

%We also emphasize that the constants $C$ in the above Theorems~\ref{th_scontr_estpolygonal1} and~\ref{th_scontr_estcurve1}
%may explode as $n\to \infty$ as the following example shows.
%
%\begin{example}\label{ex_scontr_inflen1}
%For the $n$-dimensional space $\ell^\infty_n$ equipped with the norm
%$\|\cdot\|_\infty$ defined by
%\[
%\|x\|_\infty:= \max_{i=1,\ldots, n} x^i \quad \mbox{for $x=(x^1,\ldots, x^n)$}
%\]
%define the curve $\theta\colon [0,1]\to \ell^\infty_n$ by setting
%\[
%\theta(t):=\left\{
%\begin{array}{rl}
%     0, & t\in \left[0, \frac{1}{n+1}\right),\\
%    \sum_{i=1}^{k-1} e_k, & t\in \left(\frac{k-1}{n+1}, \frac{k}{n+1}\right], \quad k=2,\ldots, n+1,
%\end{array}
%\right.
%\]
%where $\{e_k\}_{k=1}^n$ stands for the standard orthonormal basis of $\R^n$.
%Clearly, then $\theta$ is self-contracted and $\ell(\theta)=n$.
%\end{example}

The example below shows that no similar result can be expected in an infinite-dimensional situation
(even in an infinite-dimensional Hilbert space instead of the Euclidean one).

\begin{example}\label{ex_scontr_inflen2}
Let $\ell^2$ stand for the standard Hilbert space of square summable sequences equipped with its usual norm $\|\cdot\|_2$,
$\{e_k\}_{k=1}^\infty$ standing for its usual orthonormal basis.
Then the curve $\theta\colon [0,+\infty)\to \ell^2$ defined by
\[
\theta(t):=\left\{
\begin{array}{rl}
     0, & t\in \left[0, 1\right),\\
    \sum_{j=1}^{k-1} \frac{1}{j} e_j, & t\in \left[k-1, k\right), \quad k\in \N, k\geq 2,
\end{array}
\right.
\]
is self-contracted because
\begin{align*}
    \|\theta(k),\theta(l)\|_2^2 & = \sum_{j=l}^{k-1} \frac{1}{j^2} >  \sum_{j=m}^{k-1} \frac{1}{j^2}
= \|\theta(k),\theta(m)\|_2^2
\end{align*}
when $l < m < k$, $\{l,m,k\}\subset \N$,
and its trace belongs even to a compact subset of $\ell^2$ (the Hilbert cube), but
$\ell(\theta)\geq \sum_{k=1}^\infty 1/k =+\infty$.
The same curve restricted to every finite interval of time, say, $[0, n]$, $n\in \N$,
provides an example of a self-contracted curve in a bounded subset of the Euclidean space $\R^n$,
for which the constant $C$ in Theorem~\ref{th_scontr_estcurve1} tends to infinity as $n\to +\infty$.
The same example can be also easily interpreted in the language of self-contracted polygonal lines rather than curves.
It is also an easy exercise to transform this example in the one with continuous self-contracted curves.
\end{example}

It is worth providing also another simple though instructive example.

\begin{example}\label{ex_scontr_inflen3}
Let $L^2(0,1)$ stand for the standard Lebesgue space of square integrable functions over $(0,1)$ equipped with its usual
norm still denoted $\|\cdot\|_2$ (there is obviously no confusion with the previous example, though the notation for the norm is the same),
and the curve $\theta\colon [0,1]\to L^2(0,1)$  be defined by $\theta(t):= \mathbf{1}_{[0,t]}$, the characteristic function
of the interval $[0,t]$, $t\in [0,1]$. It is obviously self-contracted because
$\|\theta(t)-\theta(s)\|_2 = |t-s|^{1/2}$,
and the same relationship shows that it is not rectifiable,
though its image is a compact set (as a continuous image of $[0,1]$). Note that this is nothing but a standard construction of the isometric embedding into $L^2(0,1)$ of the ``snowflake'' space $[0,1]$ equipped with the distance $d(t,s):=|t-s|^{1/2}$.
\end{example}

%Note that the above Example~\ref{ex_scontr_inflen2} %and~\ref{ex_scontr_inflen2}
%can be also easily interpreted in the language of self-contracted polygonal lines rather than curves. On the other hand,
%it is an easy exercise to transform them in the one with continuous self-contracted curves.

The rest of the paper will be dedicated to the proof of Theorem~\ref{th_scontr_estpolygonal1}, first in an easy particular case
and then in the general situation.

\section{The heart of the proof: an easy case}\label{sec_sc_easy1}

Before presenting the quite lengthy and technical proof of Theorem~\ref{th_scontr_estpolygonal1} in its full generality, we provide here for the readers' convenience its extremely simple version for a particular situation of the two-dimensional space $\R^2$ equipped with the maximum norm $\|(x_1,x_2)\|_\infty:= |x_1|\vee |x_2|$, so that its closed unit ball
$\B$ is the square $[-1,1]^2$. This proof, though quite immediate, represents the heart of our general construction, and hence hopefully simplifies the reading of the general proof. We will further comment on how the general proof is obtained from this easy particular situation.

\begin{proof}[Proof of Theorem~\ref{th_scontr_estpolygonal1}, easy case]
The ball $\B=[-1,1]^2$ can be represented as the union of $N=4$ triangles $\{\P_i\}_{i=1}^N$, the vertices of each of the triangles $\P_i$ being the origin and two neighboring vertices of the square $[-1,1]^2$. We assume the triangles to be intersecting each other only at the origin (so they are neither open nor closed). By scaling we may assume without loss of generality that $A_j\in \B + A_r$ for all $j=1,\ldots, r$. The rest of the proof is divided then in three steps.

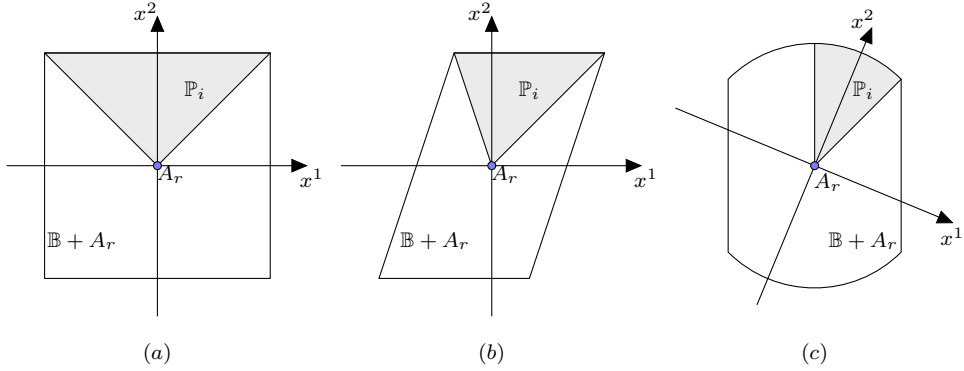
\begin{figure}
\definecolor{zzttqq}{rgb}{0.6,0.2,0.}
\definecolor{uuuuuu}{rgb}{0.26666666666666666,0.26666666666666666,0.26666666666666666}
\definecolor{xdxdff}{rgb}{0.49019607843137253,0.49019607843137253,1.}
\begin{tikzpicture}[line cap=round,line join=round,>=triangle 45,x=0.5cm,y=0.5cm]
%\clip(-4.560953035509725,-4.711504009163794) rectangle (4.897837342497144,4.780828178694151);
\fill[color=lightgray,fill=lightgray, fill opacity=0.3] (-3.,3.) -- (0.,0.) -- (3.,3.) -- cycle;
\draw (-3.,3.)-- (-3.,-3.) -- (3., -3) -- (3., 3.) -- cycle;
%\draw (3.,-3.)-- (-3.,-3.);
%\draw (-3.,-3.)-- (3.,3.);
%\draw (-3.,3.)-- (3.,-3.);
%\draw (3.,-3.)-- (3.,3.);
%\draw (3.,3.)-- (-3.,3.);
\draw [color=black] (-3.,3.)-- (0.,0.);
\draw [color=black] (0.,0.)-- (3.,3.);
\draw [color=black] (3.,3.)-- (-3.,3.);
%\draw[->,color=black] (-4.560953035509725,0.) -- (4.897837342497144,0.);
%\draw[->,color=black] (0.,-4.711504009163794) -- (0.,4.780828178694151);
\draw[->,color=black] (-4.,0.) -- (4.,0.);
\draw[->,color=black] (0.,-4) -- (0.,4);
\begin{scriptsize}
\draw [fill=xdxdff] (0.,0.) circle (1.5pt);
\draw[color=black] (0.351812943871707729,-0.2359129438717078) node {$A_r$};
\draw[color=black] (4.104601374570456,-0.3352875143184434) node {$x^1$};
\draw[color=black] (-0.30942749140893568,4.1056139747995364) node {$x^2$};
\draw[color=black] (1.,2.) node {$\P_i$};
\draw[color=black] (-2.,-2.) node {$\B+A_r$};
\draw[color=black] (0.,-5.) node {$(a)$};
\end{scriptsize}
\end{tikzpicture}
\begin{tikzpicture}[line cap=round,line join=round,>=triangle 45,x=0.5cm,y=0.5cm]
%\clip(-4.560953035509725,-4.711504009163794) rectangle (4.897837342497144,4.780828178694151);
\fill[color=lightgray,fill=lightgray, fill opacity=0.3] (-1.,3.) -- (0.,0.) -- (3.,3.) -- cycle;
\draw (-1.,3.)-- (-3.,-3.) -- (1., -3) -- (3., 3.) -- cycle;
%\draw (3.,-3.)-- (-3.,-3.);
%\draw (-3.,-3.)-- (3.,3.);
%\draw (-3.,3.)-- (3.,-3.);
%\draw (3.,-3.)-- (3.,3.);
%\draw (3.,3.)-- (-3.,3.);
\draw [color=black] (-1.,3.)-- (0.,0.);
\draw [color=black] (0.,0.)-- (3.,3.);
\draw [color=black] (3.,3.)-- (-1.,3.);
\draw[->,color=black] (-4.,0.) -- (4.,0.);
\draw[->,color=black] (0.,-4) -- (0.,4);
\begin{scriptsize}
\draw [fill=xdxdff] (0.,0.) circle (1.5pt);
\draw[color=black] (0.351812943871707729,-0.2359129438717078) node {$A_r$};
\draw[color=black] (4.104601374570456,-0.3352875143184434) node {$x^1$};
\draw[color=black] (-0.30942749140893568,4.1056139747995364) node {$x^2$};
\draw[color=black] (1.,2.) node {$\P_i$};
\draw[color=black] (-1.5,-2.) node {$\B+A_r$};
\draw[color=black] (0.,-5.) node {$(b)$};
\end{scriptsize}
\end{tikzpicture}
\begin{tikzpicture}[line cap=round,line join=round,>=triangle 45,x=0.5cm,y=0.5cm]
%\clip(-4.560953035509725,-4.711504009163794) rectangle (4.897837342497144,4.780828178694151);
\fill[color=lightgray,fill=lightgray, fill opacity=0.3] (0.,0.)--(2.3,2.3) arc (45.: 90.: {2.3*sqrt(2)})  -- (0.,0.);
\draw (2.3,-2.3)--(2.3,2.3) arc (45.: 135.: {2.3*sqrt(2)})  -- (-2.3,2.3) -- (-2.3,-2.3) arc (225.: 315.: {2.3*sqrt(2)}) -- cycle;
%\draw [color=black] (-2.3,2.3)-- (0.,0.);
\draw [color=black] (0,{2.3*sqrt(2)})-- (0.,0.);
\draw [color=black] (0.,0.)-- (2.3,2.3);
\draw[->,color=black,rotate=-22.5] (-4.,0.) -- (4.,0.);
\draw[->,color=black,rotate=-22.5] (0.,-4) -- (0.,4);
\begin{scriptsize}
\draw [fill=xdxdff] (0.,0.) circle (1.5pt);
\draw[color=black] (0.351812943871707729,-0.4359129438717078) node {$A_r$};
\draw[color=black,rotate=-22.5] (4.104601374570456,-0.3352875143184434) node {$x^1$};
\draw[color=black,rotate=-22.5] (-0.30942749140893568,4.1056139747995364) node {$x^2$};
\draw[color=black] (1.3,2.) node {$\P_i$};
\draw[color=black] (1.3,-2.) node {$\B+A_r$};
\draw[color=black] (0.,-5.) node {$(c)$};
\end{scriptsize}
\end{tikzpicture}
\caption{(a) Proof of the easy case of Theorem~\ref{th_scontr_estpolygonal1} ($n=2$, maximum norm), Step~2; (b) the case $n=2$, the unit ball is a convex polygon; (c) the case $n=2$, generic norm.}
\label{fig_sc_easy1}
\end{figure}

{\em Step 1}. We claim that
\begin{equation}\label{eq_scontr_varAikjk1_a0}
\ell(A_{1},\ldots, A_{r})\leq C_1 \sum_{i=1}^N
\ell((A_{1},\ldots, A_{r})\cap (\P_i+A_{r})) + C_2  |A_{1}A_{r}|
\end{equation}
for some positive constants $C_1$ and $C_2$.
In fact, consider an arbitrary $j\in \{1,\ldots, r-1\}$ such that $A_{j+1}\in \P_i+A_{r}$, but $A_j\not \in \P_i+A_{r}$
for some $i=1,\ldots, N$. Then for every fixed $i\in \{1, \ldots, N\}$ either
\begin{itemize}
  \item[(i)] $j$ is the first integer in $\{1, \ldots, r\}$ such that $A_{j+1}\in \P_i+A_{r}$, i.e.\
  $\{s\in \{1,\ldots, j\}\colon A_s\in \P_i+A_{r}\}=\emptyset$. In this case, since $\{A_j,A_{j+1}\}\in \lambda \B+A_{r}$
  with $\lambda:= \|A_1 A_r\|$ (because $(A_1,\ldots, A_r)$ is self-contracted), then
  \[
  |A_jA_{j+1}|\leq \sqrt{2} \|A_jA_{j+1}\| \leq 2 \sqrt{2} \|A_1A_r\| \leq 2 \sqrt{2} |A_1A_r|.
  \]
For each $i=1,\ldots, N$ except one (for which $A_1\in \P_i + A_r$) there is clearly one and only one such $j$ and hence the sum of Euclidean lengths of all such line segments $|A_jA_{j+1}|$ through all $\P_i$, $i=1,\ldots, N$, is estimated from above by $C_2 |A_{1}A_{r}|$, with $C_2 := (N-1)2\sqrt{2}=6\sqrt{2}$;
  \item[(ii)] or there is an
\[
s(j):=\max\{ s\in \{1,\ldots, j\}\colon A_s\in \P_i+A_{r}\},
\]
and $s(j)<j$ by the definition of $s(\cdot)$, hence
\[
|A_jA_{j+1}|\leq \sqrt{2} \|A_jA_{j+1}\| \leq \sqrt{2} \|A_{s(j)} A_{j+1}\|\leq  \sqrt{2}|A_{s(j)}A_{j+1}|.
\]
%(again cfr. Remark~\ref{rm_sc1subvec2})
%for $C>0$ depending only on $\|\cdot\|$ (same as $C$ in~(i)).
Then with $C_2:=\sqrt{2}$ one has
\begin{equation*}\label{eq_sc_easii1}
\begin{aligned}
\sum_{\stackrel{j\in \{1,\ldots, r-1\}}{\{A_j,A_{j+1}\}\subset \P_i+A_{r}}} & |A_jA_{j+1}| +
\sum_{\stackrel{j\in \{1,\ldots, r-1\}}{\mathrm{as~in~(ii)}}} |A_{j}A_{j+1}| \\
& \leq
C_2 \left(\sum_{\stackrel{j\in \{1,\ldots, r-1\}}{\{A_j,A_{j+1}\}\subset \P_i+A_{r}}} |A_jA_{j+1}| +
\sum_{\stackrel{j\in \{1,\ldots, r-1\}}{\mathrm{as~in~(ii)}}} |A_{s(j)}A_{j+1}|\right) \\
& = C_2 \ell((A_{1},\ldots, A_{r})\cap (\P_i+A_{r})).
\end{aligned}
\end{equation*}
%with $C_2:=\sqrt{2}$.
\end{itemize}
Hence, from~(i) and~(ii), we get~\eqref{eq_scontr_varAikjk1_a0}.
%Writing
%\begin{align*}
%\ell\left(\{A_{j}\}_{j=1}^{r}\right) %& = \sum_{j=1}^{r-1} |A_jA_{j+1}| \\
%&= \sum_{i=1}^N \left(\sum_{\stackrel{j\in \{1,\ldots, r-1\}}{\{A_j,A_{j+1}\}\subset \P_i+A_{r}}} |A_jA_{j+1}| +
%\sum_{\stackrel{j\in \{1,\ldots, r-1\}}{\mathrm{as~in~(i)}}} |A_{j}A_{j+1}|\right)
%%\\
%%&\qquad\qquad\qquad\qquad
%+ \sum_{i=1}^N
%\sum_{\stackrel{j\in \{1,\ldots, r-1\}}{\mathrm{as~in~(ii)}}} |A_{j}A_{j+1}|,  %\\
%%&\leq \sum_{i=1}^N \left(\sum_{j\in \{1,\ldots, r-1\}\colon \{A_j,A_{j+1}\}\subset \P_i+A_{r}} |A_jA_{j+1}| +
%%\sum_{j\in \{1,\ldots, r-1\}\,\mathrm{as~in~(ii)}} C_1 |A_{s(j)}A_{j+1}| \right) +
%%C_2 |A_1A_r|\\
%%&\leq C_1 \sum_{i=1}^N \left(\sum_{j\in \{1,\ldots, r-1\}\colon \{A_j,A_{j+1}\}\subset \P_i+A_{r}} |A_jA_{j+1}| +
%%\sum_{j\in \{1,\ldots, r-1\}\,\mathrm{as~in~(ii)}}  |A_{s(j)}A_{j+1}| \right) +
%%C_2 |A_1A_r|\\
%%&\leq C_1 \sum_{i=1}^N\ell((A_{1},\ldots, A_{r})\cap (\P_i+A_{r})) + C_2 |A_1A_r|.
%\end{align*}
%and estimating the first term being on the right-hand side by $C_1 \sum_{i=1}^N\ell((A_{1},\ldots, A_{r})\cap (\P_i+A_{r}))$ by~\eqref{eq_sc_easii1}, and the second by $C_1|A_1A_r|$, we get~\eqref{eq_scontr_varAikjk1_a0}.

{\em Step 2}. Assume now that $(A_{1},\ldots, A_{r})\subset \P_i+A_r$ for some $i=1,\ldots, N$, and show that
\begin{equation}\label{eq_sc_easyLx}
\ell(A_{1},\ldots, A_{r})\leq C|A_{1}A_{r}|
\end{equation}
for some universal constant $C>0$.
We consider to this aim the system of cartesian coordinates with axes passing through $A_r$ (considered then as the origin of the system), with
 the axis $x^2$ directed perpendicular to the side of $\P_i$ coinciding with a side of the square $\partial \B$ (see Figure~\ref{fig_sc_easy1}(a)), and the axis $x^1$ parallel to the latter side.
Then
\begin{equation}\label{eq_sc_easyLx0}
\ell(A_{1},\ldots, A_{r})\leq \sum_{j=1}^{r-1}\left( |x^1_{j+1}-x^1_j| +  |x^2_{j+1}-x^2_j|\right) =
\ell_{x^1}(A_{1},\ldots, A_{r})+\ell_{x^2}(A_{1},\ldots, A_{r}),
\end{equation}
where $x_j^l:= p_{x^l} (A_j)$, $l=1,2$, $j=1,\ldots, r$.
But $x^2_{j+1}\leq x^2_j$ for all $j$ (because $(A_{1},\ldots, A_{r})$ is self-contracted and $(A_{1},\ldots, A_{r})\subset \P_i+A_r$), so that
\begin{equation}\label{eq_sc_easyLx2}
\ell_{x^2}(A_{1},\ldots, A_{r})=\sum_{j=1}^{r-1}|x^2_{j+1}-x^2_j| =-\sum_{j=1}^{r-1}(x^2_{j+1}-x^2_j)= |x^2_{r}-x^2_1|.
\end{equation}

To estimate
$\ell_{x^1}(A_{1},\ldots, A_{r})$, observe that there is a subset
$(A_{q_1},\ldots, A_{q_l})\subset (A_{1},\ldots, A_{r})$ having the same $\ell_{x^1}$
but with $(x_{q_j}^1-x_{q_{j+1}}^1)(x_{q_{j-1}}^1-x_{q_j}^1) <0$, i.e.\
the projection of $(A_{q_j}A_{q_{j+1}})$ over $x^1$ is directed oppositely to that of
$(A_{q_{j-1}}A_{q_j})$, for all $j=2,\ldots, l-1$. The respective set of indices
$\Lambda:= \{{q_1},\ldots, {q_l}\}$  is formed by downward induction, namely,
setting $q_l:=r$, $q_{l-1}:=r-1$ and then for each $j$ having determined $q_j$ and $q_{j+1}$, finding the maximum index $s<j$ such that
$(x_{q_j}^1-x_{q_{j+1}}^1)(x_{s}^1-x_{q_j}^1) <0$ and setting $q_{j-1}:= s$.
Since clearly $\ell_{x^1}(A_{j-1},A_j, A_{j+1}) = \ell_{x^1}(A_{j-1},A_{j+1})$ when
%$(x_{j-1}^1-x_{j}^1)(x_{j}^1-x_{J+1}^1) \geq 0$,
the projections of $(A_{j-1}A_{j})$ and of $(A_{j}A_{j+1})$ over $x^1$ have the same direction,
then
\[
\ell_{x^1}(A_{q_1},\ldots, A_{q_l})=\ell_{x^1}(A_{1},\ldots, A_{r}),
\]
and therefore we may assume without loss of generality (up to renaming the indices) that the original
vector $(A_{1},\ldots, A_{r})$ has the property that
the projection of $(A_jA_{j+1})$ over $x^1$ is directed oppositely to that of
$A_{j-1}A_{j}$, for all $j=2,\ldots, r-1$.
Now, we note that
\[
|x^1_{j+1}-x^1_{j-1}| \vee |x^2_{j+1}-x^2_{j-1}| = \|A_{j-1}A_{j+1}\| \geq \|A_jA_{j+1}\|
%=|x^1_{j+1}-x^1_{j}|\vee  |x^2_{j+1}-x^2_{j}|
\geq |x^1_{j+1}-x^1_{j}|
\]
implies that
\begin{itemize}
\item[(i)] either $|x^1_{j+1}-x^1_j|  \leq |x^2_{j+1}-x^2_{j-1}|$
(i.e.\ the segment $(A_{j-1}A_{j+1})$ is ``vertical'' in the sense that its maximum norm is given
by the length of its projection onto $x^2$),
%|(\vec{A_{i}A_{i+1}})_x|&\leq |(\vec{A_{i-1}A_{i+1}})_y|\\
\item[(ii)] or $|x^1_{j+1}-x^1_j|  \leq |x^1_{j+1}-x^1_{j-1}|$
(i.e.\ the segment $(A_{j-1}A_{j+1})$ is ``horizontal'' in the sense that its maximum norm is given
by the length of its projection onto $x^1$).
%|(\vec{A_{i}A_{i+1}})_x|&\leq |(\vec{A_{i-1}A_{i+1}})_x|.
\end{itemize}
In case~(ii) $x_{j+1}^1$ is closer to $x_{j}^1$ than to $x_{j-1}^1$ but lies on the same side of
$x_{j}^1$ as $x_{j-1}^1$, so that
\[
|x^1_{j+1}-x^1_j| \leq \frac{1}{2}|x^1_{j}-x^1_{j-1}|,
\]
and therefore
in either of the cases
\begin{align*}
|x^1_{j+1}-x^1_j| &\leq \frac{1}{2}|x^1_{j}-x^1_{j-1}|+ |x^2_{j+1}-x^2_{j-1}|.
%=\frac{1}{2}|x^1_{j}-x^1_{j-1}|+ |x^2_{j+1}-x^2_{j}| + |x^2_{j} -x^2_{j-1}|,
\end{align*}
Thus, by induction,
\begin{align*}
|x^1_{j+1}-x^1_j| &\leq \frac{|x^1_{2}-x^1_1| }{2^{j-1}}+\frac{1}{2}\sum_{k=3}^{j+1} \frac{|x^2_{k}-x^2_{k-2}|}{2^{j-k}},
%= \frac{|x^1_{2}-x^1_1| }{2^j}+\frac{1}{2}\sum_{k=3}^{j+1} \frac{|x^2_{k}-x^2_{k-1}|+|x^2_{k-1}-x^2_{k-2}|}{2^{j-k}}\\
%%&= \frac{|x^1_{2}-x^1_1| }{2^j}+\frac{1}{2}\sum_{k=3}^{j+1} \frac{|x^2_{k}-x^2_{k-1}|}{2^{j-k}} +\frac{1}{2}\sum_{k=2}^{j}
%%\frac{|x^2_{k}-x^2_{k-1}|}{2^{j-k+1}}\\
%&= \frac{|x^1_{2}-x^1_1| }{2^j}+|x^2_{j+1}-x^2_j|+\frac{3}{4}\sum_{k=3}^{j} \frac{|x^2_{k}-x^2_{k-1}|}{2^{j-k}} + \frac{|x^2_{1}-x^2_2| }{2^j}.
\end{align*}
and therefore,
\begin{equation}\label{eq_sc_easyLx1}
\begin{aligned}
\ell_{x^1} & (A_{1},\ldots, A_{r}) =
\sum_{j=1}^{r-1}|x^1_{j+1}-x^1_j|
\leq |x^1_{2}-x^1_1| \sum_{j=1}^{r-1} \frac{1}{2^{j-1}} + \frac{1}{2} \sum_{j=1}^{r-1} \sum_{k=3}^{j+1}\frac{|x^2_{k}-x^2_{k-2}|}{2^{j-k}}\\
&= |x^1_{2}-x^1_1| \sum_{j=1}^{r-1} \frac{1}{2^{j-1}} + \sum_{k=3}^{r}|x^2_{k}-x^2_{k-2}| \sum_{j=k-1}^{r-1} \frac{1}{2^{j-k+1}}\\
&\leq  2 |x^1_{2}-x^1_1| + 2 \sum_{k=3}^{r}|x^2_{k}-x^2_{k-2}|\leq
2|x^1_{2}-x^1_1|+2\ell_{x^2}(A_{1},\ldots, A_{r})%(1-2^{-(r+1)})
\\
& \leq 2|x^1_{2}-x^1_1|+2|x^2_{1}-x^2_r| %(1-2^{-(r+1)})
\quad\mbox{by~\eqref{eq_sc_easyLx2}}\\
&\leq 2 |A_{1}A_2|+2|A_{1}A_r|\leq 4\sqrt{2}|A_{1}A_r| + 2|A_1A_r|
\quad\mbox{since $\{A_1,A_2\}\subset A_r+\|A_1A_r\|\B$}.
%\\
%&=
%2(\sqrt{2}+1) |A_{1}A_r|.
\end{aligned}
\end{equation}
Plugging~\eqref{eq_sc_easyLx1} and~\eqref{eq_sc_easyLx2} into~\eqref{eq_sc_easyLx0}, we get~\eqref{eq_sc_easyLx} as claimed.

{\em Step 3}. Denoting now
$(A_{j^1_i},\ldots, A_{j^{m(i)}_i}):=(A_{1},\ldots, A_{r})\cap (\P_i+A_{r})$ for each $i=1,\ldots, N$, where
$j^l_i\in \{1, \ldots, r\}$ (and, clearly, $j^{m(i)}_i=r$), one has applying the result of Step~2 (with
$j^1_i$ instead of $1$ and $j^{m(i)}_i$ instead of $r$) the estimate
\[
\ell((A_{1},\ldots, A_{r})\cap (\P_i+A_{r})) \leq C |A_{j^1_i}A_{j^{m(i)}_i}| = C |A_{j^1_i}A_{r}|\leq C\sqrt{2} |A_{1}A_{r}|,
\]
and hence applying~\eqref{eq_scontr_varAikjk1_a0}, we conclude the proof.
\end{proof}

\subsection*{From an easy particular case to the general situation}

The above argument captures all the essential features of the complete proof: namely, the division of the unit ball $\B$ of the norm in a finite number of ``cone-like'' subsets $\P_i$, the reduction of the estimate of the variation of $(A_1, \ldots, A_r)$ to the estimates of variations of its subvectors in each $\P_i$ (i.e.\ Step~1 of the above proof, cfr.\ Step~1 of the proof of the key Lemma~\ref{lm_scontr_ind0} in the sequel), and
the separate estimate of the variation of a polygonal line inside each $\P_i$ along each of the appropriately chosen axes
with a separate consideration of ``vertical'' and ``horizontal'' line segments (i.e.\ Step~2 of the above proof).
However, there are several substantial difficulties arising on this way, which we list below.
\begin{itemize}
  \item[(i)] The first difficulty comes already when trying to generalize the above argument to the case $n=2$ and the unit ball $\B$ of the norm $\|\cdot\|$ an arbitrary convex \emph{polygon} rather than a square. The division of $\B$ into triangles $\P_i$ and the choice of the axes for each $\P_i$ is natural and clearly the same as in our simple model case of the maximum norm, see Figure~\ref{fig_sc_easy1}(b), so that the self-contracted polygonal line with vertices inside $\P_i$ with this choice never goes upwards in the direction of $x^2$. However,
      the argument of Step~2 of the above model proof becomes much more involved because the estimates for ``horizontal'' and ``vertical'' parts of the polygonal line are not at all that simple as in the case of the maximum norm.
  \item[(ii)] The next difficulty comes with the case of a generic norm $\|\cdot\|$ (i.e.\ with the unit ball not necessarily a polygon) in $\R^2$ (i.e.\ still $n=2$). In this case it is only possible to choose the division of the unit ball $\B$ of the norm $\|\cdot\|$ into the sets $\P_i$ such that  with some natural choice of the direction $x^2$ for each $\P_i$ the self-contracted polygonal line with vertices inside $\P_i$ might go upwards in the direction of $x^2$ but ``not too much'', see Figure~\ref{fig_sc_easy1}(c). This naturally leads to quantitative notions of the ``horizontality'' and ``verticality'' for line segments.
  \item[(iii)] However, the major difficulty comes when trying to adapt these arguments to the generic space dimension $n$. In fact, consider even the simplest case, say, of the maximum norm $\|\cdot\|$ in $\R^3$, with the unit ball $\B$ of the norm being the cube $[-1,1]^3$. The sets $\P_i$ then are (quite similarly to the case of a maximum norm in $\R^2$) the six pyramids with one of the vertices in the origin, the bases being the faces of the cube (which in a sense justifies the notation $\P_i$ for them). With the natural choice of coordinates for each of such pyramids ($x^3$ perpendicular to the base) the self-contracted polygonal line with vertices inside the pyramid can only go downwards in the direction of $x^3$, but has a lot of freedom in the other two directions.
      Hence, intuitively, one has to consider separately the ``horizontal'' parts of this polygonal line making separate estimates for their subparts belonging again to different pyramids (of course, related to the pyramid $\P_i$ originally considered), with vertices shifted away from the origin. This leads to the notion of \emph{admissible} ordered sets of pyramids (see Definition~\ref{def_sc_admpyr1}), each such set producing a natural system of not necessarily orthogonal coordinates,  and to a technically involved inductive argument for the generic space dimension.
      Very roughly speaking, in a generic space dimension $n$, in each of the sets $\P_i$ we will calculate separately the variation of the ``vertical'' part of the self-contracted vector, which is easy by monotonicity (or ``almost monotonicity'' in the case of a generic norm, when $\P_i$ is no more a pyramid) of $x^n$ coordinates of the vertices, the axis $x^n$ being determined by the set $\P_i$, and then, by induction, the variation of its ``horizontal'' part. When calculating the latter, we will arrange this part in subparts belonging to different sets $\P_j$, each one determining the respective axis $x^{n-1}$, and estimate again separately the variation of its ``vertical'' and  ``horizontal'' parts
      (now with respect to $x^{n-1}$), proceeding by backward induction.

      Note that as explained above, in principle one can avoid using induction for $n=2$; however, the general proof we provide uses induction even in this relatively simple case.
      Last but not least, it is worth mentioning that for a generic norm $\|\cdot\|$ the sets $\P_i$ are no more pyramids and may have a quite complicated structure (although this would not affect the proof).
\end{itemize}

\section{Preliminary constructions}

\subsection{Partition of a convex body}\label{sec_scontr_partconv1}
For a convex set $\D\subset \R^n$  we say that $\nu_x\in\R^n$ is a vector of external normal to $\D$ at $x\in \partial \D$,
if there is a support hyperplane $\Pi$ to $\D$ at $x$ orthogonal to $\nu_x$ and $\nu_x$ is directed towards the open half-space
bounded by $\Pi$ and not containing $\D$. Clearly, an $x\in \partial \D$ may have many external normal vectors, unless $\partial \D$ is smooth.

We will need the following construction.

\begin{proposition}\label{prop_part_convbody1}
Let $\D\subset \R^n$ be a %bounded
convex set. Then for every $\delta>0$ there is a cover of $\partial\D$ by a finite number of sets
$\{\T_i\}_{i=1}^N$ (with some $N=N(\delta)\in \N$) with the following property:
for every $i\in \{1,\ldots, N\}$ there is a vector $\nu^i$ such that for every $x\in \T_i$ there is a vector of external normal
$\nu_x$ to $\D$ at $x$ with $\widehat{(\nu_x, \nu^i)}<\delta$.
\end{proposition}

\begin{proof}
%Without loss of generality we may assume $\D$ to be $n$-dimensional (i.e.\ not contained in a hyperplane).
Given a $\delta>0$, we take a finite cover of $S^{n-1}$ as in the statement being proven, i.e.\
find $\{\tilde \T_i\}_{i=1}^N$ (with some $N=N(\delta)\in \N$), $\cup_{i=1}^N \tilde \T_i=S^{n-1}$ such that
for every $i\in \{1,\ldots, N\}$ there is a vector $\nu^i$ with the property that for every $x\in \tilde \T_i$ the unique vector of external unit normal
$\nu_x$ to $S^{n-1}$ at $x$ satisfies $\widehat{(\nu_x, \nu^i)}<\delta$.
It suffices to define now $\T_i$ to be the set of all
$x\in \partial\D$ that admit a unit vector of external normal $\nu_x$ to $\D$ coinciding with some vector of external unit normal
$\nu_y$ to $S^{n-1}$ at some $y\in \tilde \T_i$.
\end{proof}

\begin{remark}\label{rm_convbody1_disj1}
Although the sets $\T_i$ mentioned in the above Proposition~\ref{prop_part_convbody1} may be overlapping, we may easily make them disjoint by substituting $\T_i$ with $\T_i\setminus \cup_{j=1}^{i-1} \T_j$.
\end{remark}

Applying the above Proposition~\ref{prop_part_convbody1} with Remark~\ref{rm_convbody1_disj1} to the closed unit ball $\B$ of $\|\cdot\|$, given a $\delta>0$, we find an $N=N(\delta)\in \N$ and a  finite family of disjoint sets
$\{\T_i\}_{i=1}^N$ and vectors $\{\nu^i\}_{i=1}^N$ such that $\cup_i \T_i=\partial \B$ and
for every $i\in \{1,\ldots, N\}$ and every $x\in \T_i$ there is a vector of external unit normal
$\nu_x$ to $\D$ at $x$ with $\widehat{(\nu_x, \nu^i)}<\delta$. Define then $\P_i:=\cup_{t\in [0,1]} t\T_i$, and note
that now by construction one has $\P_i\cap \P_j=\{0\}$ for $i\neq j$.
We will further frequently use the following fact.

\begin{lemma}\label{lm_scontr_epsconvA1Ar}
Let %$(A_1, \ldots, A_r)$
$(A_1, \ldots, A_r)\subset \P_i+A_r$ for some $i\in \{1,\ldots, N\}$
be self-contracted with respect to the norm $\|\cdot\|$.
Then $\widehat{((A_jA_{j+1}), (\nu^i)^\bot)} > \delta$, $j\in \{1,\ldots, r-1\}$ implies
\[(A_j-A_{j+1})\cdot \nu^i >0.
\]
\end{lemma}

\begin{proof}
The result follows from Lemma~\ref{lm_scontr_epsconv1} below applied with
$\D:= A_r +\lambda \B$, $\T:=A_r+\lambda \T_i$, where $\lambda:=\|A_j-A_r\|$, $A_j$ in place of $A_1$ and $A_{j+1}$ in place of
$A_2$, once we observe that $A_{j+1}\in \D$ by self-contracted requirement on $(A_1, \ldots, A_r)$.
\end{proof}

The following immediate geometric fact has been used in the above proof.

\begin{lemma}\label{lm_scontr_epsconv1}
Let $\D\subset \R^n$ be a convex set, $\delta>0$, $\T\subset \partial\D$, $\nu\in \R^n$ be such that for every $x\in \T$ there is a vector of external normal
$\nu_x$ to $\D$ at $x$ with $\widehat{(\nu_x, \nu)}<\delta$. %Let $\P:=\cup_{t\in [0,1]} t\T$.
If $A_1\in \T$, $A_2\in \D$ and $\widehat{((A_1A_2), \nu^\bot)} > \delta$, then $(A_1-A_2)\cdot \nu >0$.
\end{lemma}

\begin{proof}
If $B\in \R^n$ satisfying $\widehat{((A_1 B), \nu^\bot)} > \delta$ is such that $(B-A_1)\cdot \nu \geq 0$, then
$\widehat{(B-A_1, \nu)} < \pi/2-\delta$, so that
\[
\widehat{(B-A_1, \nu_{A_1})} \leq  \widehat{(B-A_1, \nu)} + \widehat{(\nu, \nu_{A_1})}<
\frac \pi 2 - \delta + \delta =\frac \pi 2.
\]
Hence
$B$ belongs to the open
half-space bounded by the hyperplane $A_1+\nu_{A_1}^\bot$ and not containing $\D$ (see Figure~\ref{fig_epsconv1}).
Thus, it is impossible that $(A_2-A_1)\cdot \nu \geq 0$ because $A_2\in D$, which concludes the proof.
\end{proof}

\begin{figure}
\definecolor{qqwuqq}{rgb}{0.,0.39215686274509803,0.}
\definecolor{xdxdff}{rgb}{0.49019607843137253,0.49019607843137253,1.}
%\definecolor{aqaqaq}{rgb}{0.6274509803921569,0.6274509803921569,0.6274509803921569}
\begin{tikzpicture}[line cap=round,line join=round,>=triangle 45,x=1.0cm,y=1.0cm]
\clip(-5.92,-1.74) rectangle (4.04,4.76);
\draw [rotate around={1.48305457747507:(-0.61,0.52)},
color=lightgray,fill=lightgray,
fill opacity=0.3] (-0.61,0.52) ellipse (3.4623380389565095cm and 1.5598989377537318cm);
\draw [shift={(-0.7965625037975763,2.0741983593184274)},
color=qqwuqq,fill=qqwuqq,
fill opacity=0.1] (0,0) -- (62.57446143194419:0.6) arc (62.57446143194419:92.57446143194416:0.6) -- cycle;
\draw [shift={(-0.7965625037975763,2.0741983593184274)},color=qqwuqq,fill=qqwuqq,fill opacity=0.1] (0,0) -- (152.21893859424912:0.6) arc (152.21893859424912:182.2189385942491:0.6) -- cycle;
\draw [shift={(-0.7789060564479413,2.0749922477863127)},color=qqwuqq,fill=qqwuqq,fill opacity=0.1] (0,0) -- (-27.634057793212346:0.6) arc (-27.634057793212346:2.5744614319441705:0.6) -- cycle;
\draw [->] (-0.7965625037975763,2.0741983593184274) -- (0.1,3.78);
\draw [->] (-0.7965625037975763,2.0741983593184274) -- (-0.8829434449557151,3.9953505126340874);
\draw (-3.863086539587659,3.6896979191036885)-- (1.58,0.84);
\draw (-0.7965625037975763,2.0741983593184274)-- (-1.9,3.64);
\draw [shift={(-0.4543906656477419,-4.858479001211346)},line width=2.pt]  plot[domain=1.2921454832096615:1.947295684264658,variable=\t]({1.*6.941116405307211*cos(\t r)+0.*6.941116405307211*sin(\t r)},{0.*6.941116405307211*cos(\t r)+1.*6.941116405307211*sin(\t r)});
\begin{scriptsize}
\draw (2.,2.9) node[anchor=north west] {$A_1+\nu^\perp_{A_1}$};
\draw (-5.60171506509224,1.8581438422742096)-- (3.8389826290210443,2.282626801939603);
\draw[color=black] (-1.28,0.36) node {$\D$};
\draw [fill=xdxdff] (-0.7965625037975763,2.0741983593184274) circle (1.5pt);
\draw[color=black] (-0.96,1.85) node {$A_1$};
\draw[color=black] (-0.10,3.14) node {$\nu$};
\draw[color=black] (-0.6,2.89) node {$<\! \delta$};
\draw[color=black] (-1.64,2.26) node {$<\! \delta$};
\draw[color=black] (-3.42,3.92) node {$A_1+\nu^{\perp}$};
\draw [fill=xdxdff] (-1.9,3.64) circle (1.5pt);
\draw[color=black] (-1.76,3.92) node {$B$};
\draw[color=black] (-1.09,3.22) node {$\nu_{A_1}$};
\draw[color=black] (0.58,1.82) node {$\T$};
\end{scriptsize}
\end{tikzpicture}
\label{fig_epsconv1}
\caption{Construction of the proof of Lemma~\ref{lm_scontr_epsconv1}.}
\end{figure}
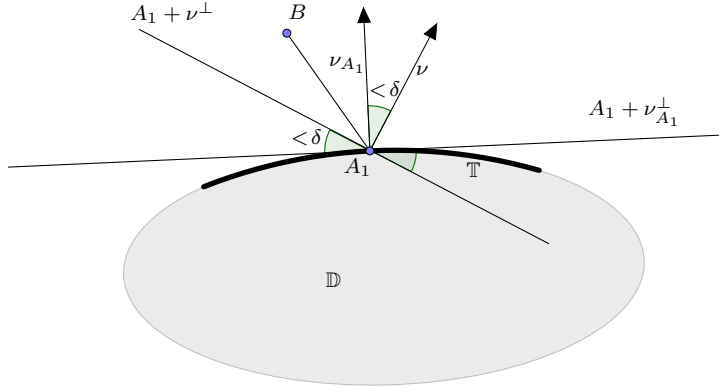

For a linear subspace $\Pi\subset \R^n$ of arbitrary dimension, %$0\in \Pi$,
and an $\varepsilon\geq 0$ we define
\[
V_\varepsilon(\Pi):=\{z\in \R^n\colon \widehat{((0,z),\Pi)}\leq \varepsilon\}\cup\{0\},
\]
We also need the following statement.

\begin{lemma}\label{lm_scontr_defxi1}
Let $\D\subset \R^n$ be a bounded convex set with nonempty interior $\mathrm{int}\,\D$, and $0\in \mathrm{int}\,\D$. Then there is an $\varepsilon_0>0$ and
%$\varepsilon_0>0$ such that for every $\varepsilon\in (0,\varepsilon_0)$
%there is
a $\bar\xi\in (0, \pi/2)$ (depending %on $\varepsilon$ and
only on $\D$) such that for every linear subspace $\Pi\subset \R^n$ of arbitrary dimension, %$0\in \Pi$,
and every $x\in V_{\varepsilon_0}(\Pi)\cap \partial\D$ one has that every external normal $\nu_x$ to $\D$ at $x$ satisfies
$\widehat{(\nu_x,\Pi)} < \bar\xi$.
%where
%\[
%V_\varepsilon(\Pi):=\{z\in \R^n\colon \widehat{(z,\Pi)}\leq \varepsilon\}.
%\]
\end{lemma}

\begin{proof}
If the claim is false, then there is a sequence of linear subspaces $\{\Pi_k\}$,
%(which may be considered of fixed dimension) %ok ma non serve,
%$0\in \Pi_k$,
and of points $\{x_k\}\subset\partial\D$ such that $\lim_k\widehat{(\nu_{x_k},\Pi_k)} =\pi/2$.
Up to passing to subsequences (not relabeled), we may assume that $\Pi_k\to \Pi$ in the sense of Hausdorff (for some linear subspace $\Pi\subset \R^n$),
%$0\in \Pi$),
$x_k\to x\in \partial\D$, $\nu_{x_k}\to \nu$ (for some $\nu\in \R^n$) and support hyperplanes
$x_k + (\nu_{x_k})^\bot$
to $\D$ at $x_k$ orthogonal to $\nu_{x_k}$ converge to $x+\nu^\bot$ which is necessarily a support hyperplane to $\D$ at $x$
(as limit of support hyperplanes) as $k\to \infty$. Since then  $\widehat{(\nu,\Pi)} =\pi/2$, one has $\Pi\subset \nu^\bot$, and hence
$0\in \nu^\bot$, which means $0\in \partial\D$, contradicting the assumption $0\in \mathrm{int}\,\D$.
\end{proof}

We apply the above Lemma~\ref{lm_scontr_defxi1} to $\D:=\B$ (the closed unit ball %, sometimes called also \emph{Wulff shape},
of $\|\cdot\|$), and find an $\varepsilon_0>0$ and a $\bar \xi$ depending %on $\varepsilon$ and
only on $\B$ (hence on $\|\cdot\|$) such that for every linear subspace $\Pi\subset \R^n$ of arbitrary dimension
%passing through the origin,
and every $x\in V_{\varepsilon_0}(\Pi)\cap \partial\D$ one has that every external normal $\nu_x$ to $\B$ at $x$ satisfies
$\widehat{(\nu_x,\Pi)} <\bar \xi< \pi/2$. Hence denoting
\[
\mathbb{A}_\varepsilon (\Pi):=\{i\in \{1,\ldots, N\}\colon \T_i\cap V_\varepsilon(\Pi)\neq\emptyset\},
\]
we have that $\widehat{(\nu^i,\Pi)} <\bar\xi+\delta$ for all $i\in \mathbb{A}_{\varepsilon_0} (\Pi)$.
Thus there is a $\bar\delta >0$ depending only on $\bar\xi$ (hence only %on $\varepsilon$ and
on $\|\cdot\|$)
such that  $\widehat{(\nu^i,\Pi)} <\xi$ for all $i\in \mathbb{A}_{\varepsilon_0} (\Pi)$
and for some $\xi\in (0,\pi/2)$, which still depends only on %$\varepsilon$ and
$\|\cdot\|$, when $\delta <\bar\delta$
(one may take, say, $\xi := \bar\xi/2 +\pi/4$ and $\bar\delta := \pi/4 - \bar\xi/2$). The following notion will be at the heart of our inductive construction.

\begin{definition}\label{def_sc_admpyr1}
%Chosen an $\varepsilon\in (0, \varepsilon_0]$,
We will call the (ordered) $(n-i+1)$-tuple of sets $(\P_{\alpha_i},\ldots, \P_{\alpha_n})$, with $\alpha_j\in \{1,\ldots, N\}$, $j=i, \ldots, n$, $i\in \{1,\ldots, n\}$
\emph{admissible}, if for all $j=i+1, \ldots, n$ one has
\[
\alpha_{j-1}\in \mathbb{A}_{\varepsilon_0} (\Pi^{j-1}),\quad\mbox{where } \Pi^{j-1}:=(\mathrm{span}\,\{\nu^{\alpha_k}\}_{k=j}^n)^\bot,
\]
or, equivalently, $\T_{\alpha_{j-1}}\cap V_{\varepsilon_0}(\Pi^{j-1})\neq \emptyset$, that is,
\begin{equation}\label{eq_scontr_admpyram1}
    \widehat{(\nu^{\alpha_{j-1}}, \Pi^{j-1})} < \xi\quad\mbox{for all $j=i+1, \ldots, n$}.
\end{equation}
In case $i=n$ the ordered $(n-i+1)$-tuple $(\P_{\alpha_i},\ldots, \P_{\alpha_n})$ reduces to a singleton which is by definition always considered admissible.
\end{definition}

Each admissible ordered $n$-tuple of sets $(\P_{\alpha_1},\ldots, \P_{\alpha_n})$ will determine in a natural way a (not necessarily orthogonal) coordinate system (different for different admissible $n$-tuples): in fact, the
axes
$x^j$, $j=2, \ldots, n$, will be directed along vectors $\nu^{\alpha_j}$ determined by the sets
$\P_{\alpha_j}$, while
the axis $x^1$ will be chosen orthogonal to all $x^j$ with $j=2, \ldots, n$.
The idea of the proof of Theorem~\ref{th_scontr_estpolygonal1} is then as follows: the whole vector $(A_1,\ldots, A_r)$ will be appropriately arranged in subvectors naturally corresponding
to some admissible ordered $n$-tuples $(\P_{\alpha_1},\ldots, \P_{\alpha_n})$
so that the total variation of each of subvectors could be evaluated in a coordinate system determined by the respective
$n$-tuple, and the total variation of the whole vector would then be estimated by the sum of variations of the chosen subvectors.

Finally, we will need the following simple geometric lemma.

\begin{lemma}\label{lm_scontr_estproj1}
Let $\Pi\subset\R^n$ be a $k$-dimensional linear subspace,
 $\{\nu_i\}_{i=1}^k\subset\Pi$ such that
% $\{x^i\}_{i=1}^k$ be arbitrary lines belonging to a $k$-dimensional hyperplane $\Pi\subset\R^n$,
%with
\begin{equation}\label{eq_sc_angxi1}
\widehat{\left(x^i,\left(\mathrm{span}\{\nu_j\}_{j=i+1}^k\right)^\bot\right)}\leq \zeta, \quad i=1, \ldots, k-1
\end{equation}
for some $\zeta\in (0, \pi/2)$,
where $x^i:=\mathrm{span}\{\nu_i\}$, $i=1,\ldots, k$.
%$\widehat{(x^i,x^j)}\geq \xi >0$ whenever $i\neq j$.
Then there is a constant $C=C_k(\zeta)>0$ %%% after submit
such that for every $A\in \R^n$ one has
\[
|p_\Pi(A)|\leq C\sum_{i=1}^k |p_{x^i}(A)|.
\]
\end{lemma}

\begin{proof}
Note first that according to the condition~\eqref{eq_sc_angxi1} on $\{\nu_j\}$, these vectors are linearly independent and hence
$\Pi=\mathrm{span}\,  \{\nu_i\}_{i=1}^k$.
One further notes that it is enough to prove the statement for $A\in \Pi$, which in this case reduces to
\begin{equation}\label{eq_scontr_estproj1a}
|A|\leq C_k(\zeta)\sum_{i=1}^k |p_{x^i}(A)|.
\end{equation}
In fact then for an arbitrary $A\in \R^n$ one has
%,
%denoting for brevity
%$A_i':=p_{x^i}(p_\Pi(A))$,
the estimate
\[
|p_\Pi(A)|\leq C_k(\zeta) \sum_{i=1}^k |p_{x^i}(p_\Pi(A))|\leq C_k(\zeta) \sum_{i=1}^k |p_{x^i}(A)|,
\]
the latter inequality being valid because $%|A_i'|=
|p_{x^i}(p_\Pi(A))|= |p_{x^i}(A)|$.

We therefore prove~\eqref{eq_scontr_estproj1a} for $A\in \Pi$. For this purpose we use the (finite) induction on $k$. The statement is trivial for $k=1$ (with $C_1(\zeta):=1$). Suppose it is true for $k=m$. To prove it for $k=m+1$, for an arbitrary linear subspace
$\Pi\subset \R^n$ of dimension $m+1$ we denote
\begin{align*}
\Pi_m &:=\left(\mathrm{span}\,  \{\nu_j\}_{j=2}^{m+1}\right)^\bot\cap \Pi,\\
\Pi_m^\bot &:=\mathrm{span}\,  \{\nu_j\}_{j=2}^{m+1}.
\end{align*}
Note that for $A\in \Pi$ one has
\begin{equation}\label{eq_scontrPx1pm1}
\begin{aligned}
    p_{x_1}(A) & = p_{x_1}(p_{\Pi_m}(A) + p_{\Pi_m^\bot}(A)) = p_{x_1}(p_{\Pi_m}(A)) + p_{x_1}(p_{\Pi_m^\bot}(A))\\
    & =p_{\Pi_m}(A)\cos\widehat{(x^1, \Pi_m)} + p_{\Pi_m^\bot}(A)\sin\widehat{(x^1, \Pi_m)},
\end{aligned}
\end{equation}
and thus
\begin{align*}
    |A|& \leq |p_{\Pi_m}(A)| + |p_{\Pi_m^\bot}(A)| \\
& = \frac{1}{\cos\widehat{(x^1, \Pi_m)}}\left(|p_{x^1}(A)|- |p_{\Pi_m^\bot}(A)|\sin\widehat{(x^1, \Pi_m)} \right)+ |p_{\Pi_m^\bot}(A)|\quad\mbox{by~\eqref{eq_scontrPx1pm1}}\\
&\leq \left(1+\tan \widehat{(x^1, \Pi_m)}+ \frac{1}{\cos\widehat{(x^1, \Pi_m)}}\right)\left(|p_{x^1}(A)| + |p_{\Pi_m^\bot}(A)| \right)\\
&\leq \left(1+\tan \zeta+ \frac{1}{\cos\zeta}\right)\left(|p_{x^1}(A)| + |p_{\Pi_m^\bot}(A)| \right) \quad\mbox{because $\widehat{(x^1, \Pi_m)}\leq\zeta$}\\
& \leq \left(1+\tan \zeta+ \frac{1}{\cos\zeta}\right)\left(|p_{x^1}(A)| + C_m(\zeta)\sum_{i=2}^{k+1} |p_{x^i}(A)|\right)
\quad\mbox{by inductive assumption},\\
\end{align*}
which implies the claim with $C_{m+1}(\zeta):= \left(1+\tan \zeta+ \frac{1}{\cos\zeta}\right) (C_m(\zeta)\vee 1)$.
\end{proof}

From now on we denote
\begin{equation}\label{eq_scontr_estproj1b}
    C(\zeta):=\max_{k\in \{1,\ldots, n\}} C_k(\zeta),
\end{equation}
where $C_k(\zeta)>0$ is defined by Lemma~\ref{lm_scontr_estproj1} (in fact, we may clearly always take
$C_k(\zeta)$ nondecreasing with $k$, in which case
$C(\zeta)=C_n(\zeta)$). Clearly, for a fixed $\zeta$ the constant $C(\zeta)$ depends only on the space dimension $n$.

\subsection{Preliminary lemmata on self-contracted polygonal lines}

In the sequel we will extensively use without any further reference the following immediate observations.

\begin{remark}\label{rm_sc1subvec1}
If $(A_1,\ldots,A_r)$ is self-contracted, %(resp.\ strictly self-contracted),
then so is any its subvector
$(A_{j_1},\ldots,A_{j_k})$.
\end{remark}

\begin{remark}\label{rm_sc1subvec2}
If $(A_1,\ldots,A_r)$ is self-contracted, %(resp.\ strictly self-contracted),
then
\[
|A_jA_m|\leq C|A_1A_r|
\]
for all $(j,m)\subset \{1,\ldots, r\}$ and for some $C>0$ depending only on  $\|\cdot\|$.
In fact, the self-contracted property of $(A_1,\ldots,A_r)$ implies $\{A_j, A_m\}\subset A_r +\lambda \B$, where $\B$ is the closed unit ball of $\|\cdot\|$ and $\lambda = |A_1A_r|$, so that $|A_jA_m|\leq \lambda\diam \B$, hence one may take $C:=\diam \B$.
\end{remark}

We will also need the following lemmata.

\begin{lemma}\label{lm_scontr_horiz1}
There is a constant $\varepsilon_1\in (0,\pi/2)$ depending only on the norm $\|\cdot\|$ such that
when the vector $(A_1,A_2,A_3)$  is self-contracted with respect to
the norm $\|\cdot\|$,
 $\angle A_1 A_2 A_3 \leq 2\varepsilon_1$,
then $|A_2A_3|\leq 3/4 |A_1 A_2|$.
\end{lemma}

\begin{proof}
Let $\B$ as usual stand for the closed unit ball of the norm $\|\cdot\|$.
By Lemma~\ref{lm_scontr_defxi1} the minimum over $P\in \partial \B$ angle between $(OP)$ and any support hyperplane to $\B$
at $P\in \partial \B$, is at least $\pi/2-\bar\xi >0$.
For an arbitrary $\alpha\in (0,\pi/2-\bar\xi)$ denote
\[
\Delta(\alpha):=\inf\{|PQ|\colon \{P, Q\}\subset \partial \B, P\neq Q, \angle OPQ\leq \alpha\},
\]
see Figure~\ref{fig_scontr_horiz1b}(A).
We have that $\Delta(\alpha)>0$. In fact, otherwise there is a sequence $\{P_\nu, Q_\nu\}_\nu\subset \partial \B$, $P_\nu\neq Q_\nu$, $\lim_\nu P_\nu =\lim_\nu Q_\nu =P\subset \partial \B$, $\angle OP_\nu Q_\nu\leq \alpha$. Choose an arbitrary support hyperplane $\Pi_\nu$ to $\B$ at $Q_\nu$. Without loss of generality we may assume the existence of a limit $\lim_\nu \Pi_\nu=\Pi$, and thus $\Pi$ is a support hyperplane to $\B$ at $P$. Denoting $R_\nu:= \Pi_\nu\cap (O P_\nu)$, we get
\begin{align*}
    \widehat{((OP),\Pi)} & =\lim_\nu\widehat{((OP_\nu),\Pi_\nu)} \leq \lim_\nu\angle OR_\nu Q_\nu\\
& \leq \lim_\nu\angle OP_\nu Q_\nu\quad\mbox{because $P_\nu\in [OR_\nu]$},\\
& \leq \alpha,
\end{align*}
see Figure~\ref{fig_scontr_horiz1b}(B),
which is impossible for $\alpha\in (0,\pi/2-\bar\xi)$, because $\widehat{((OP),\Pi)}\geq \pi/2-\bar\xi$.

\definecolor{uuuuuu}{rgb}{0.26666666666666666,0.26666666666666666,0.26666666666666666}
\begin{figure}
\begin{tabular}{cc}
\begin{minipage}[h]{0.49\linewidth}
\definecolor{zzttqq}{rgb}{0.6,0.2,0.}
\definecolor{uuuuuu}{rgb}{0.26666666666666666,0.26666666666666666,0.26666666666666666}
\definecolor{qqwuqq}{rgb}{0.,0.39215686274509803,0.}
\definecolor{qqqqff}{rgb}{0.,0.,1.}
\begin{tikzpicture}[line cap=round,line join=round,>=triangle 45,x=2.05cm,y=2.05cm]
\clip(-1.4215536869531458,1.863129896154361) rectangle (3.065301680919299,5.280042830149537);
\draw [shift={(1.1197277368122516,4.2199659671015315)},color=qqwuqq,fill=qqwuqq,fill opacity=0.1] (0,0) -- (-133.26656578969553:0.517714080908359) arc (-133.26656578969553:-108.09920109297994:0.517714080908359) -- cycle;
\fill[color=lightgray,%
%zzttqq,
fill=lightgray,%
%zzttqq,
fill opacity=0.1] (1.1197277368122516,4.2199659671015315) -- (0.5955563162011475,4.13593102252658) -- (0.28433282934843307,3.705864177342698) -- (0.3683677739233848,3.1816927567315934) -- (0.7984346191072668,2.870469269878879) -- (1.3226060397183712,2.9545042144538307) -- (1.6338295265710854,3.3845710596377128) -- (1.5497945819961338,3.9087424802488173) -- cycle;
\draw (1.1197277368122516,4.2199659671015315)-- (2.3398979013045897,3.3364265456608013);
\draw (1.1197277368122516,4.2199659671015315)-- (0.3632671582529751,3.4162904140669275);
\draw (1.1197277368122516,4.2199659671015315)-- (0.922314237095856,3.615950085082242);
\draw [color=zzttqq] (1.1197277368122516,4.2199659671015315)-- (0.5955563162011475,4.13593102252658);
\draw [color=zzttqq] (0.5955563162011475,4.13593102252658)-- (0.28433282934843307,3.705864177342698);
\draw [color=zzttqq] (0.28433282934843307,3.705864177342698)-- (0.3683677739233848,3.1816927567315934);
\draw [color=zzttqq] (0.3683677739233848,3.1816927567315934)-- (0.7984346191072668,2.870469269878879);
\draw [color=zzttqq] (0.7984346191072668,2.870469269878879)-- (1.3226060397183712,2.9545042144538307);
\draw [color=zzttqq] (1.3226060397183712,2.9545042144538307)-- (1.6338295265710854,3.3845710596377128);
\draw [color=zzttqq] (1.6338295265710854,3.3845710596377128)-- (1.5497945819961338,3.9087424802488173);
\draw [color=zzttqq] (1.5497945819961338,3.9087424802488173)-- (1.1197277368122516,4.2199659671015315);
\draw (1.1197277368122516,4.2199659671015315)-- (-0.732865600751041,3.922958972213191);
\begin{scriptsize}
%\draw [fill=xdxdff] (0.3632671582529751,3.4162904140669275) circle (2.5pt);
%\draw [fill=xdxdff] (0.3532671582529751,3.4062904140669275) circle (2.5pt);
\draw [fill=xdxdff] (0.3452671582529751,3.3932904140669275) circle (2.5pt);
\draw[color=black] (0.2075302875620538044,3.457872043303342) node {$Q$};
\draw [fill=xdxdff] (0.922314237095856,3.615950085082242) circle (2.5pt);
\draw[color=black] (1.03797595815583705,3.6406149072730636) node {$O$};
\draw[color=qqwuqq] (0.81699310852253056,3.7061006352125062) node {$\alpha$};
\draw [fill=xdxdff] (1.1197277368122516,4.2199659671015315) circle (2.5pt);
\draw[color=black] (1.2360452617097637,4.33689574367264395) node {$P$};
\draw[color=black] (1.1325024455280919,3.364500730788605) node {$B$};
\end{scriptsize}
\end{tikzpicture}
\end{minipage}
%\hfill
&
\begin{minipage}[h]{0.49\linewidth}
\begin{tikzpicture}[line cap=round,line join=round,>=triangle 45,x=.8cm,y=.8cm]
\clip(-4.64,-5.58) rectangle (4.8,1.34);
\draw (-4.,0.)-- (4.,0.);
\draw [shift={(0.,-5.)}] plot[domain=0.6435011087932844:2.49145814292177,variable=\t]({1.*5.*cos(\t r)+0.*5.*sin(\t r)},{0.*5.*cos(\t r)+1.*5.*sin(\t r)});
\draw (0.,-5.)-- (2.56,0.7);
\draw (0.,-5.)-- (0.,0.);
\begin{scriptsize}
\draw (3.1,-1.72) node[anchor=north west] {$\partial \B$};
\draw [fill=xdxdff] (0.,-5.) circle (2.5pt);
\draw[color=black] (-0.4,-4.78) node {$O$};
\draw [fill=xdxdff] (0.,0.) circle (2.5pt);
\draw[color=black] (0.14,0.28) node {$Q_\nu$};
\draw [fill=xdxdff] (2.245614035087719,0.) circle (2.5pt);
\draw[color=black] (2.7,0.24) node {$R_\nu$};
%\draw [fill=uuuuuu] (3.6245614035087719,0.) circle (1.5pt);
\draw[color=black] (3.84,0.24) node {$\Pi_\nu$};
\draw [fill=xdxdff] (2.0484956320222016,-0.43889644432556646) circle (2.5pt);
\draw[color=black] (2.44,-0.4) node {$P_\nu$};
\end{scriptsize}
\end{tikzpicture}
\end{minipage}\\
(A) & (B)
\end{tabular}
\caption{Constructions in the proof of Lemma~\ref{lm_scontr_horiz1}. (A): Definition of $\Delta(\alpha)$. (B): Proof of $\Delta(\alpha)>0$.}
\label{fig_scontr_horiz1b}
\end{figure}
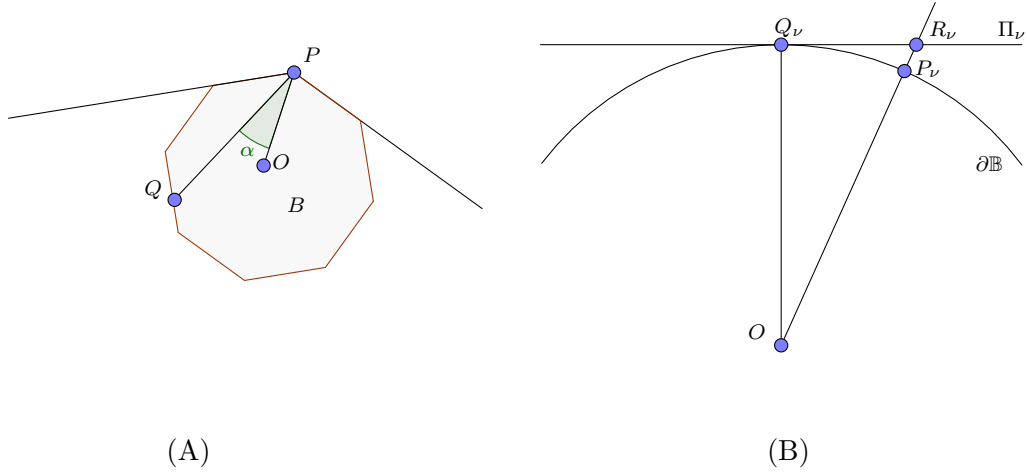

Consider the set $\D_3:=\{z\in \R^n\colon \|A_2-z\| \leq \|A_1-z\|\}$ (see Figure~\ref{fig_scontr_horiz1}) and
observe that $\{A_2, A_3\}\subset \D_3$.
By Lemma~\ref{lm_scontr_mediatr1} there is an $\bar{\varepsilon}>0$ such that
\begin{equation}\label{eq_scontr_cal1}
|A_2 C_\alpha|\leq 3/4 |A_1 A_2| \quad\mbox{when $\angle A_1A_2 C_\alpha\leq 2\bar{\varepsilon}$, $C_\alpha\in M(A_1,A_2)$,}
\end{equation}
where $M(A_1,A_2)$ stands for the \emph{mediatrix} of the segment $[A_1 A_2]$ (i.e.\ the set of points equidistant from the endpoints of the segment).
 Fix an arbitrary
$\alpha\in (0,\pi/2-\bar\xi)$ and
let $\varepsilon_1 := \bar{\varepsilon} \wedge \alpha/2$, so that
$\Delta(2\varepsilon_1)\geq \Delta(\alpha)>0$ by definition of $\Delta(\cdot)$.
We can find an $r\geq \|A_2 -A_3\|$ such that $r\Delta(2\varepsilon_1) > |A_1 A_2|$, thus
for every
$\{P, Q\}\subset \partial (r\B +C)$, $C\in \R^n$,  $P\neq Q$, $\angle CPQ\leq 2\varepsilon_1$ one has
$|PQ|>|A_1A_2|$.
Suppose that $\angle A_1 A_2 A_3 \leq 2\varepsilon_1$.
Consider the point $C\in [A_2 A_3)$ with $\| A_2-C\|=r$. One has $A_2\in \partial (r\B+C)$. Thus if
$Q\in (A_1A_2) \cap \partial (r\B+C)$, $Q\neq A_2$ and is on the line $(A_1A_2)$ on the same side of $A_2$ as $A_1$, then $|A_2Q|> |A_1A_2|$, which implies $A_1\in r\B +C$,
i.e. $\|A_1-C\| < r= \|A_2-C\|$, or, in
other words, $C\in \D_3^c$.
Thus there is a point $C'\in [A_3, C]$ such that
$C'\in M(A_1,A_2)$, see Figure~\ref{fig_scontr_horiz1}.  Therefore,
\begin{align*}
    |A_2 A_3|\leq  |A_2 C'|\leq 3/4 |A_1 A_2|,
\end{align*}
the last inequality being due to~\eqref{eq_scontr_cal1}.
\end{proof}

\begin{figure}[h]
\definecolor{zzttqq}{rgb}{0.6,0.2,0.}
\definecolor{qqqqff}{rgb}{0.,0.,1.}
\definecolor{ffvvqq}{rgb}{1.,0.3333333333333333,0.}
%\definecolor{xdxdff}{rgb}{0.49019607843137253,0.49019607843137253,1.}
\begin{tikzpicture}[line cap=round,line join=round,>=triangle 45,x=.8cm,y=.8cm]
\clip(-6.160724282766859,-2.0913898964255115) rectangle (0.58179294396102584,3.8371842689794757);
\fill[line width=1.6pt,color=lightgray,%
%zzttqq,
fill=lightgray,%
%zzttqq,
fill opacity=0.1] (0.7522994672994994,-5.744791688390951) -- (-6.915534509489739,-6.045674767887297) -- (-2.638582435422924,-0.8881737373949808) -- (-2.638582435422924,-0.013107439918230152) -- (-2.638582435422924,0.8881737373949815) -- (0.6509975114486717,4.828805264080135) -- cycle;
\draw (0.,0.)-- (-7.582707016166586,3.9465675561640703);
\draw [rotate=-10] (0.,0.)-- (-7.74,0.);
\begin{scriptsize}
\draw [color=black](-1.5634630999647638,2.480831507890512) node[anchor=north west] {$M(A_1, A_2$)};
\draw [domain=-6.160724282766859:-2.638582435422924] plot(\x,{(-3.4120573161959524-1.7938859098273405*\x)/-1.4876127057104833});
\draw [domain=-2.638582435422924:0.8179294396102584] plot(\x,{(--6.015576801126841--1.7938859098273405*\x)/1.4876127057104833});
\draw [line width=1.6pt,color=zzttqq] (0.7522994672994994,-5.744791688390951)-- (-6.915534509489739,-6.045674767887297);
\draw [line width=1.6pt,color=zzttqq] (-6.915534509489739,-6.045674767887297)-- (-2.638582435422924,-0.8881737373949808);
\draw [line width=1.6pt,color=zzttqq] (-2.638582435422924,-0.8881737373949808)-- (-2.638582435422924,-0.013107439918230152);
\draw [line width=1.6pt,color=zzttqq] (-2.638582435422924,-0.013107439918230152)-- (-2.638582435422924,0.8881737373949815);
\draw [line width=1.6pt,color=zzttqq] (-2.638582435422924,0.8881737373949815)-- (0.6509975114486717,4.828805264080135);
\draw [line width=1.6pt,color=zzttqq] (0.6509975114486717,4.828805264080135)-- (0.7522994672994994,-5.744791688390951);
\draw [fill=xdxdff] (-5.164134378080573,2.6877743197452006) circle (2.5pt);
\draw[color=black] (-5.001261438610159,3.0715012586873187) node {$C$};
\draw [fill=xdxdff] (-2.339781346801269,1.2177847742391124) circle (2.5pt);
\draw[color=black] (-2.4198158610537335,1.6932718401614364) node {$C'$};
\draw [fill=xdxdff, rotate=-10] (-4.726195141133405,0.) circle (2.5pt);
\draw[color=black, rotate=-10] (-3.929305224201135,0.4681790236939853) node {$A_1$};
\draw [fill=xdxdff] (-1.2131408571162743,0.6314027876251229) circle (2.5pt);
\draw[color=black] (-1.019709785090926,1.0807254319277109) node {$A_3$};
\draw [fill=xdxdff] (0.,0.) circle (2.5pt);
\draw[color=black] (0.2053830313765303,0.3806723939463102) node {$A_2$};
\draw [fill=xdxdff] (0.6509975114486717,4.828805264080135) circle (2.5pt);
\draw[color=black] (1.9555156263300393,3.6621710094841253) node {$Q$};
%\draw[color=zzttqq] (-3.032362269287461,-5.351011854526407) node {$o$};
\draw[color=black] (-1.135419314373787,-0.9100503948318972) node {$\D_3$};
\end{scriptsize}
\end{tikzpicture}
\caption{Proof of the claim of Lemma~\ref{lm_scontr_horiz1}.}
\label{fig_scontr_horiz1}
\end{figure}

\begin{corollary}\label{co_scontr_horiz2}
Assume that the vector $(A_1,\ldots, A_r)$  is self-contracted with respect to
the norm $\|\cdot\|$ and
has alternating directions along some axis $x=\mathrm{span}\{\nu\}$
for some $\nu\in \R^n$, i.e.\ the finite sequence of numbers $\{(A_{k+1}-A_k)\cdot \nu\}_{k=1}^{r-1}$
has alternating signs.
If
each segment $[A_k A_{k+1}]$, $k=1, \ldots, r-1$, is
$\varepsilon_1$-horizontal
with respect to $x$, %the $x$ axis,
where $\varepsilon_1>0$ is as in Lemma~\ref{lm_scontr_horiz1},
then
\begin{eqnarray}
\ell(A_1,\ldots, A_r)& \leq 4|A_1A_2|. \label{eq_sumhoriz1}
%\ell_{x}(A_1,\ldots, A_r)& \leq 5|A_1^1A_2^1|, \label{eq_sumhoriz1} %ok ma semplificato
%\\
%\ell_1(A_1,\ldots, A_r)& \leq 80|A_r^1-A_1^1|. \label{eq_sumhoriz2} %%% ok ma non serve
\end{eqnarray}
%where $A_i^1:=p_{x^1}(A_i)$, $i=1,2$.
\end{corollary}

\begin{proof}
%Let $\varepsilon_1$ be as in Lemma~\ref{lm_scontr_horiz1}.
%and take an $\varepsilon_1\leq \varepsilon$ so that additionally
%$\cos\varepsilon_1\geq 4/5$.
By triangle inequality for $\widehat{((A_{k-1}A_k), x)}$ and $\widehat{((A_kA_{k+1}), x)}$ %%% after submit
we get $\angle A_{k-1}A_kA_{k+1} \leq 2\varepsilon_1$.
Hence
applying Lemma~\ref{lm_scontr_horiz1} to each consecutive triple $(A_{k-1}, A_k, A_{k+1})$ yields
\[
|A_k A_{k+1}|\leq 3/4 |A_{k-1} A_k|, \quad k=2,\ldots, r-1.
\]
Thus,
\begin{align*}
\ell(A_1,\ldots, A_r)  &= \sum_{k=1}^{r-1} |A_kA_{k+1}| \leq |A_1A_2| \sum_{k=1}^{r-1} \left(\frac 3 4\right)^{k-1}
%\\
%&
\leq 4|A_1 A_2|
%\leq 4 \frac{|A_1^1A_2^1|}{\cos\varepsilon_1}\leq 5|A_1^1A_2^1| %%% ok ma semplificato
\end{align*}
proving~\eqref{eq_sumhoriz1}.
%%%% dim di~\eqref{eq_sumhoriz1}: ok ma non serve.
%At last,
%\begin{equation}\label{eq_esthor_32a}
%|A_{k+1}^1-A_k^1| \leq |A_k A_{k+1}|\leq \frac{3}{4} |A_{k-1} A_k| \leq \frac{3}{4} \frac{|A_{k}^1-A_{k-1}^1|}{\cos\varepsilon_1}\leq \frac{15}{16}|A_{k}^1-A_{k-1}^1|,
%\end{equation}
%and in particular, recalling that
%\[
%|A_r^1-A_1^1| = \sum_{k=1}^{r-1} (-1)^{k+1} |A_{k+1}^1-A_k^1|
%\]
%(since the segments of $(A_1,\ldots, A_r)$
%have alternating directions along $x^1$),
%we get
%\[
%|A_r^1-A_1^1| \geq |A_2^1-A_1^1| - |A_3^1-A_2^1|,
%\]
%because the terms of the above sum are decreasing by~\eqref{eq_esthor_32a}.
%Thus,
%\begin{align*}
%|A_r^1-A_1^1|
%               & %= |A_2^1-A_1^1| - |A_3^1-A_2^1|
%                 \geq |A_2^1-A_1^1| - \frac{15}{16} |A_3^1-A_1^1| \quad\mbox{by~\eqref{eq_esthor_32a}}\\
%               &= \frac{1}{16} |A_2^1-A_1^1| \geq \frac{1}{80} \ell_1(A_1,\ldots, A_r) \quad\mbox{by~\eqref{eq_sumhoriz1}},
%\end{align*}
%proving~\eqref{eq_sumhoriz2}.
\end{proof}

\section{Inductive construction}

From now on we consider the constants $\varepsilon_0>0$ defined by Lemma~\ref{lm_scontr_defxi1} and $\varepsilon_1>0$ defined by Lemma~\ref{lm_scontr_horiz1}.
% and set
%\[
%\varepsilon:=\varepsilon_0\wedge\varepsilon_1.
%\]
Let also $\bar\delta$ and $\xi$ be as defined in Section~\ref{sec_scontr_partconv1}, the constant $C(\xi)$
be defined by~\eqref{eq_scontr_estproj1b} %Lemma~\ref{lm_scontr_estproj1}
and set
\[
\delta_0:=\bar\delta %\wedge \varepsilon
\wedge \arctan \frac{\sin\varepsilon_0}{8(n-1)
C(\xi)}
%\wedge\arctan \frac{1}{2 C(\xi)} %%% garantito dalla diseq precedente
\wedge \arctan \left(\frac{1}{8(n-1)(3\cot\varepsilon_1+ 8/\sin\varepsilon_1)C(\xi)}\right).
\]
The proof of Theorem~\ref{th_scontr_estpolygonal1} will be based on the following assertion.

\begin{proposition}\label{prop_scontr_indmain1}
Assume that $\delta\in (0, \delta_0)$.
Then
there exists a constant $C>0$ depending only on $\|\cdot\|$ and on $\delta$ such that
whenever the vector
$(A_1,\ldots, A_r)\subset \R^n$
is self-contracted with respect to
the norm $\|\cdot\|$ and
$(A_1,\ldots, A_r)\subset \P_{\alpha_n} +A_r$
for some $\alpha_n\in \{1,\ldots, N\}$,
 then
\begin{equation}\label{eq_scontr_estv1_ind1a}
\ell(A_1,\ldots, A_r)\leq C|A_1A_r|,
\end{equation}
with $C>0$ depending only on $\|\cdot\|$, $n$, $\delta$ and $\xi$.
\end{proposition}

\begin{remark}\label{rem_scontr_homot1}
Note that both the self-contracted property and~\eqref{eq_scontr_estv1_ind1a} are stable under scalings and translations.
\end{remark}

Taking for the moment this result for granted, we may easily prove Theorem~\ref{th_scontr_estpolygonal1}.

\begin{proof}[Proof of Theorem~\ref{th_scontr_estpolygonal1}]
Up to scaling we may assume that the closed unit ball $\B$ of the norm $\|\cdot\|$ in $\R^n$ satisfies
$\min_{x\in \partial \B} |x|\geq 1$.
We identify now $\R^n$ with the subset $\R^n\times\{0\}$ of $\R^{n+1}$ and write every element
$x\in \R^{n+1}$ as a pair $\tilde x=(x', x_{n+1})$, $x'\in \R^n$, $x_{n+1}\in \R$, so that every element of
$x\in \R^n$ is identified with some $(x,0)\subset \R^{n+1}$.
Equip $\R^n\times [0,1]$ with the norm defined so that its closed unit ball be
$\tilde \B:=\B\times [-1,1]$. %, where $\B$ is the Wulff shape of the norm $\|\cdot\|$ in $\R^n$.
Define the decomposition of this
closed unit ball $\tilde \B$ as follows. If the family of sets $\{\T_j\}_{j=1}^N$ is the disjoint cover of $\partial \B\subset \R^n$ defined by Proposition~\ref{prop_part_convbody1}, we define the disjoint cover of $\partial \tilde \B$ by the family consisting of sets
$\tilde\T_j :=\T_j\times (-1,1)$, $j=1, \ldots, N$, and two additional sets $\tilde \T_{N+1}:= \B\times \{1\}$ and $\tilde \T_{N+2}:= \B\times \{-1\}$. Clearly, this cover satisfies the property defined by Proposition~\ref{prop_part_convbody1}.
We set then $\tilde\P_i:=\cup_{t\in [0,1]} t\T_i$.
Now we consider $(A_1, \ldots, A_r)$ as a subset of $\R^{n+1}$. Recalling Remark~\ref{rem_scontr_homot1}, without loss of generality we may scale all $A_j$ so as to have $|A_1A_r|=1$ and then shift them by a fixed vector so as to have
$A_r=(0,1)$ and all $A_j$ belong to the hyperplane $\{x_{n+1}=1\}$. Clearly, the vector
$(A_1, \ldots, A_r,0)$ is still self-contracted with respect to the introduced norm in $\R^{n+1}$ and $(A_1, \ldots, A_r,0)\subset \tilde\P_{N+1}$ by construction. Applying Proposition~\ref{prop_scontr_indmain1} with
$(A_1, \ldots, A_r,0)$ instead of $(A_1, \ldots, A_r)$,
$\R^{n+1}$ (with the introduced norm) instead of $\R^n$ (with the original norm $\|\cdot\|$), $n+1$ instead of $n$, $\tilde\P_j$ instead of $\P_j$ and
$\alpha_{n+1}:= N+1$, we get
\[
\ell(A_1,\ldots, A_r)\leq \ell(A_1,\ldots, A_r,0)\leq C|A_1|\leq   C\sqrt{2} = C\sqrt{2}|A_1A_r|
\]
with $C>0$ as claimed.
\end{proof}

The proof of Proposition~\ref{prop_scontr_indmain1} is the immediate application of the following inductive statement
with $i:=n$.

\begin{proposition}\label{prop_scontr_indmain0}
Assume that $\delta\in (0, \delta_0)$.
Then for every $i\in \{1, \ldots, n\}$
there exists a constant $C_{i}>0$ depending only on $\|\cdot\|$, $n$, $\delta$ and $\xi$ such that
whenever $(n-i+1)$-tuple $(\P_{\alpha_i},\ldots, \P_{\alpha_n})$ is admissible, the vector
$%(A_1,\ldots, A_r, A_{r+1}, \ldots, A_{r+n-i})\subset \R^n
(A_j)_{j=1}^{r+n-i}\subset \R^n$
is self-contracted with respect to
the norm $\|\cdot\|$ and
$(A_1,\ldots, A_r)\subset \P_{\alpha_j} +A_{r+n-j}$
for all $j=i,\ldots, n$,  then
%\begin{equation}\label{eq_scontr_estv1_ind1a}
\[
\ell(A_1,\ldots, A_r)\leq C_i|A_1A_r|.
\]
%\end{equation}
%with $C_i>0$ depending only on $\|\cdot\|$, $n$, $\delta$ and $\xi$.
\end{proposition}

The proof of the above Proposition~\ref{prop_scontr_indmain0} is just a (finite) induction on $i\in \{1, \ldots, n\}$ with the base of induction
given by Lemma~\ref{lm_scontr_base0} and the inductive step given by Lemma~\ref{lm_scontr_ind0} provided in the sequel. %%% after submit

\subsection{Base of induction}

We prove the following statement that will serve as a base of induction.

\begin{lemma}\label{lm_scontr_base0}
Let
\[
\delta\in \left(0, \bar\delta
%\wedge \arctan (1/4(n-1)C(\xi))
\wedge \arctan \left(\frac{1}{8(n-1)(3\cot\varepsilon_1+ 8/\sin\varepsilon_1)C(\xi)}\right)
\right),
\]
where $\varepsilon_1>0$ is defined by Lemma~\ref{lm_scontr_horiz1}
(this is true in particular when $\delta\in (0,\delta_0)$),
and
the $n$-tuple $(\P_{\alpha_1},\ldots, \P_{\alpha_n})$ with
$\alpha_j\in \{1,\ldots, N\}$, $j=1, \ldots, n$,
 be admissible.
Assume that the vector
$(A_1,\ldots, A_r, A_{r+1}, \ldots A_{r+n-1})\subset \R^n$
is self-contracted with respect to
the norm $\|\cdot\|$ and
$(A_1,\ldots, A_r)\subset \P_{\alpha_j} +A_{r+n-j}$
for all $j=1,\ldots, n$. Then
\begin{equation}\label{eq_scontr_estv1_base1}
\ell(A_1,\ldots, A_r)\leq C_1|A_1A_r|
\end{equation}
for some constant $C_1>0$ depending only on the norm $\|\cdot\|$, $n$, $\delta$ and $\xi$.
\end{lemma}

The rest of this section will be dedicated to the proof of the above Lemma~\ref{lm_scontr_base0}. To this aim we consider the coordinate system with the origin in $A_r$ and axes
$x^j$ directed along vectors $\nu^{\alpha_j}$ determined by the sets
$\P_{\alpha_j}$ with
$\alpha_j\in \{1,\ldots, N\}$, $j=2, \ldots, n$, and
$x^1:= (\mbox{span}\, \{\nu^{\alpha_j}\}_{j=2}^n)^\bot$
(with arbitrarily chosen direction).
For brevity we denote $A_i^j:=p_{x^j}(A_i)$.

\subsubsection{Common lemmata}

We will need a couple of assertions that will serve both for the base of induction and for the inductive step.

\begin{lemma}\label{lm_scontr_estvarPi-i0}
Assume that $\delta\in (0, \bar\delta)$,
and that for some %$i\in \{2, \ldots, n-1\}$
$i\in \{2, \ldots, n\}$
 the
$(n-i+2)$-tuple of sets $(\P_{\alpha_{i-1}},\ldots, \P_{\alpha_n})$ is admissible.
%the vector
%$(A_1,\ldots, A_r, A_{r+1}, \ldots A_{r+n-i+1})$
%is self-contracted with respect to
%the norm $\|\cdot\|$ and
%$(A_1,\ldots, A_r)\subset \P_{\alpha_j} +A_{r+n-j}$
%for all $j=i-1,\ldots, n$.
Then for every vector $(A_1,\ldots, A_r)$ one has
\[
\ell_{(\Pi^{i-1})^\bot}(A_1,\ldots, A_r) \leq C(\xi)\sum_{j=i}^n \ell_{x_j}(A_1,\ldots, A_r),
\]
where, as usual,
$x_j= \mathrm{span}\, \{\nu^{\alpha_j}\}$, $j=i,\ldots, n$, and
$\Pi^{i-1}=(\mathrm{span}\, \{\nu^{\alpha_j}\}_{j=i}^n)^\bot$.
\end{lemma}

\begin{proof}
Applying Lemma~\ref{lm_scontr_estproj1} with
$\Pi:=(\Pi^{i-1})^\bot$, %$k:=n-i$,
 $k:=n-i+1$,
 $\nu_j:=\nu^{\alpha_j}$, $j=i,\ldots,n$ and $\zeta:=\xi$ (the conditions of Lemma~\ref{lm_scontr_estproj1} are satisfied
in view of~\eqref{eq_scontr_admpyram1}, recalling that $\delta <\bar\delta$), we get
\[
|p_{(\Pi^{i-1})^\bot}(A)| \leq C(\xi) \sum_{j=i}^n |p_{x^j}(A)|
\]
for every $A\in \R^n$, and plugging in the latter inequality $A:=A_{m+1}-A_m$, $m=1,\ldots, r-1$, and summing over such $m$,
we get the claim.
\end{proof}

\begin{lemma}\label{lm_scontr_vert1a}
Let $(A_1,\ldots, A_r)$
be such that for some $\nu\in \R^n$, denoting by $x_k:= (A_k-A_r)\cdot \nu$, one has
$x_k<x_{k-1}$ whenever $\widehat{((A_{k-1}A_k), \nu^\bot)}>\delta$,
Then
\begin{equation}\label{eq_scontr_estv1_base0j-tan}
\ell_{x}(A_1,\ldots, A_r) \leq  |A_1A_r| + 2 \ell_{x^\bot}(A_1,\ldots, A_r)\tan\delta,
\end{equation}
where $x=\mathrm{span}\,\{\nu\}$.
\end{lemma}

\begin{proof}
For $j\in \{2, \ldots, n\}$ we have
\begin{align*}
    x_1-x_r &=-\sum_{k=2}^r (x_k-x_{k-1})= -\sum_{k\in \{2,\ldots, r\}, x_k>x_{k-1}} |x_k-x_{k-1}| +
    \sum_{k\in \{2,\ldots, r\}, x_k<x_{k-1}} |x_k-x_{k-1}| \\
    &= -\ell_{x}(A_1,\ldots, A_r) + 2\sum_{k\in \{2,\ldots, r\}, x_k>x_{k-1}} |x_k-x_{k-1}|,
\end{align*}
so that
\begin{align*}
    \ell_{x}(A_1,\ldots, A_r) &\leq |x_r-x_1| + 2\sum_{k\in \{2,\ldots, r\}, x_k>x_{k-1}} |x_k-x_{k-1}|\\
       &\leq |x_r-x_1| + 2\sum_{k\in \{2,\ldots, r\}, x_k>x_{k-1}} |p_{x^\bot}(A_k)-p_{x^\bot}(A_{k-1})|\tan \delta.
\end{align*}
Thus,
\begin{align*}
\ell_{x}(A_1,\ldots, A_r) &\leq  |x_r-x_1| + 2\ell_{x^\bot}(A_1,\ldots, A_r)\tan\delta\\
 &\leq  |A_1A_r| + 2\ell_{x^\bot}(A_1,\ldots, A_r)\tan\delta
\end{align*}
proving the claim.
\end{proof}

We will also need the following easy calculation.

\begin{lemma}\label{lm_scontr_sumLj1}
If for some $i\in \{1,\ldots, n\}$ one has
\begin{equation}\label{eq_sumLj1}
L_j\leq L_0 + C\sum_{k=i}^n L_k, \quad j=i, \ldots, n,
\end{equation}
with some $C\in (0,(n-i+1)^{-1}/2)$, then
$\sum_{j=i}^n L_j \leq 2(n-i+1)L_0$ and
$L_j\leq 2 L_0$ for all $j=i,\ldots, n$.
%$L_j\leq C_0L_0$ for all $j=i,\ldots, n$ and
%for some $C_0>0$ depending only on $C$, $n$, $i$.
\end{lemma}

\begin{proof}
Summing~\eqref{eq_sumLj1} over $j=i, \ldots, n$, we obtain
\[
\sum_{j=i}^n L_j\leq (n-i+1) L_0 + (n-i+1)C\sum_{j=i}^n L_j,
\]
so that in view of $(n-i+1)C\leq 1/2$ one has
\[
\sum_{j=i}^n L_j\leq 2(n-i+1) L_0,
\]
and hence the statement is proven by plugging the latter relationship
into~\eqref{eq_sumLj1}.
\end{proof}

\subsubsection{Estimate of the variation of $(A_1,\ldots, A_r)$ along each $x^j$, $j=2,\ldots, n$}.

\begin{lemma}\label{lm_scontr_vert1}
Assume that
\begin{equation}\label{eq_sc_vert1a}
\delta\in \left(0, \bar\delta\wedge \arctan \frac{1}{4(n-1)C(\xi)}\right)
\end{equation}
(this is true in particular when $\delta$ is as in Lemma~\ref{lm_scontr_base0}),
the $n$-tuple $(\P_{\alpha_1},\ldots, \P_{\alpha_n})$ with
$\alpha_j\in \{1,\ldots, N\}$, $j=1, \ldots, n$, and
the vector
$(A_1,\ldots, A_r, A_{r+1}, \ldots A_{r+n-1})\subset \R^n$
satisfy conditions of Lemma~\ref{lm_scontr_base0}. Then one has
\begin{equation}\label{eq_scontr_estv1_base0j}
\ell_{x^j}(A_1,\ldots, A_r) \leq  2  |A_1A_r| + 4\ell_{x^1}(A_1,\ldots, A_r)\tan\delta, \quad j=2,\ldots, n.
\end{equation}
\end{lemma}

\begin{proof}
Note first that in the particular case when $\delta$ is as in Lemma~\ref{lm_scontr_base0}, then
the trivial inequality $3\cos\varepsilon_1+8 %\geq 5
\geq (\sin\varepsilon_1)/2$ (with $\varepsilon_1>0$ defined by Lemma~\ref{lm_scontr_horiz1}) implies
$3\cot\varepsilon_1+ 8/\sin\varepsilon_1\geq 1/2$,
and therefore
%\[ %%%%% ok ma abbreviato
%\arctan \frac{1}{8 (n-1) (3\cot\varepsilon_1+ 8/\sin\varepsilon_1) C(\xi)}\leq \arctan \frac{1}{4 (n-1)  C(\xi)},
%\]
%i.e.
$\delta$ satisfies~\eqref{eq_sc_vert1a}.

For a generic $\delta$ satisfying~\eqref{eq_sc_vert1a},
since
$(A_1,\ldots, A_r)\subset \P_{\alpha_j} +A_{r+n-j}$ for $j\in \{2,\ldots, n\}$ as requested by Lemma~\ref{lm_scontr_base0}, then $(A_1,\ldots, A_r)$ satisfies
conditions of Lemma~\ref{lm_scontr_vert1a} (with $\nu^{\alpha_j}$ instead of $\nu$) in view of Lemma~\ref{lm_scontr_epsconvA1Ar}, and hence by Lemma~\ref{lm_scontr_vert1a} one has
\begin{equation}\label{eq_scontr_estxj1}
\begin{aligned}
\ell_{x^j}(A_1,\ldots, A_r)
 &\leq  |A_1A_r| + 2\ell_{(x^j)^\bot}(A_1,\ldots, A_r)\tan\delta\leq |A_1A_r| + 2\ell(A_1,\ldots, A_r)\tan\delta\\
&\leq |A_1A_r| + 2\ell_{x^1}(A_1,\ldots, A_r)\tan\delta + 2 \ell_{(x^1)^\bot}(A_1,\ldots, A_r)\tan\delta.
\end{aligned}
\end{equation}
By Lemma~\ref{lm_scontr_estvarPi-i0} with $i:=2$, we get
\[
\ell_{(x^1)^\bot}(A_1,\ldots, A_r) \leq C(\xi)\sum_{k=2}^n \ell_{x_k}(A_1,\ldots, A_r).
\]
Thus~\eqref{eq_scontr_estxj1} becomes
\begin{equation}\label{eq_scontr_estxj2}
\begin{aligned}
\ell_{x^j}(A_1,\ldots, A_r)
&\leq |A_1A_r| + 2\ell_{x^1}(A_1,\ldots, A_r)\tan\delta + 2\tan\delta C(\xi)\sum_{k=2}^n \ell_{x_k}(A_1,\ldots, A_r).
\end{aligned}
\end{equation}
We may now apply Lemma~\ref{lm_scontr_sumLj1} with
\begin{align*}
L_0& :=|A_1A_r| + 2\ell_{x^1}(A_1,\ldots, A_r)\tan\delta,\\
L_j &:=\ell_{x_j}(A_1,\ldots, A_r),
\end{align*}
 $i:=2$ and $C:=2\tan\delta C(\xi)$ (recalling $\tan\delta <1/(4(n-1)C(\xi))$)
to get the claim.
\end{proof}

\subsubsection{Estimate of the variation of $(A_1,\ldots, A_r)$ along $x^1$}

Everywhere in this subsection we denote for brevity of notation $A_j^\bot:= p_{(x^1)^\bot}(A_j)$.

\begin{lemma}\label{lm_scontr_horiz4}
Assume that the vector $(A_1,\ldots, A_r)$
is self-contracted with respect to
the norm $\|\cdot\|$.
Then for an arbitrary line $x^1$ one has
\begin{equation}\label{eq_scontr_estv1_1a}
\ell_{x^1}(A_1,\ldots, A_r)\leq C_1  \ell_{(x^1)^\bot}(A_1,\ldots, A_r) +  C_2 |A_1A_r|,
\end{equation}
with some positive constants $C_1$ and $C_2$  depending only on the norm $\|\cdot\|$, and in particular, one can take
$C_1:= 3\cot\varepsilon_1+ 8/\sin\varepsilon_1$, where $\varepsilon_1>0$ is defined by Lemma~\ref{lm_scontr_horiz1},
and $C_2:=\diam \B$, where $\B$ is the closed unit ball of $\|\cdot\|$.
\end{lemma}

\begin{proof}
Let $\varepsilon_1>0$ be as in Lemma~\ref{lm_scontr_horiz1}.
%, and $\varepsilon\in (0, \varepsilon_1]$. %ok ma non serve
Starting from the last  line segment $[A_{k-1}A_k]$ $\varepsilon_1$-horizontal with respect to $x^1$ (i.e.\ with maximum $k\in  \{2,\ldots, r\}$
such that $[A_{k-1}A_k]$ is horizontal)
we find the minimum $q\in \{1, k-1\}$ such that
$(A_q, A_{q+1}, \ldots, A_k)$ satisfies conditions of Lemma~\ref{lm_scontr_hordir1} (with $x^1$ instead of $x$ and $\varepsilon:=\varepsilon_1$). Now we repeat the same operation
with  $q-1$
instead of $r$, and continue in this way by backward induction as far as possible. In this way we find a
finite sequence of disjoint intervals $\{q_i,q_i+1,\ldots, k_i\}_{i=1}^{\nu}$ of $\{1,\ldots, r\}$
such that $(A_{q_i}, A_{q_i+1}, \ldots, A_{k_i})$ satisfies conditions of Lemma~\ref{lm_scontr_hordir1} (again with $x^1$ instead of $x$ and $\varepsilon:=\varepsilon_1$), while the
subvector
\[
(A_j)_{j\in \Lambda}
%(\tilde{A}_1,\ldots, \tilde{A}_\rho):=( A_1, \ldots, A_{q_1}, A_{k_1}, \ldots, A_{q_2}, A_{k_2},\ldots, A_{q_\nu}, A_{k_\nu}, A_r)
\subset (A_1,\ldots, A_r)
\]
obtained from $(A_1,\ldots, A_r)$ by canceling all the points $A_j$ with $j\in \cup_{i=1}^{\nu}\{q_i+1,\ldots, k_i-1\}$,
satisfies conditions of Lemma~\ref{lm_scontr_horiz3}.
Therefore, denoting for convenience $(\tilde{A}_1,\ldots, \tilde{A}_\rho):=(A_j)_{j\in \Lambda}$, we get
\begin{align*}
    \ell_{x^1}(A_1,\ldots, A_r) & = \ell_{x^1}(\tilde{A}_1,\ldots, \tilde{A}_\rho) - \sum_{i=1}^{\nu} |A_{q_i}^1A_{k_i}^1| + \sum_{i=1}^{\nu} \ell_{x^1}\left(\left(A_j\right)_{j=q_i}^{k_i}\right)\\
&\leq \ell_{x^1}(\tilde{A}_1,\ldots, \tilde{A}_\rho) \\
&\quad +2\cot\varepsilon_1 \sum_{i=1}^{\nu} \ell_{(x^1)^\bot}\left(\left(A_j\right)_{j=q_i}^{k_i}\right)\quad\mbox{by Lemma~\ref{lm_scontr_hordir1} with $x^1$ instead of $x$}\\
& \leq C  \ell_{(x^1)^\bot}(\tilde{A}_1,\ldots, \tilde{A}_\rho) +  |\tilde{A}_1\tilde{A}_2| \\
& \qquad\qquad\qquad + 2\cot\varepsilon_1 \sum_{i=1}^{\nu} \ell_{(x^1)^\bot}\left(\left(A_j\right)_{j=q_i}^{k_i}\right)\quad\mbox{by Lemma~\ref{lm_scontr_horiz3}}\\
%%%% ok ma semplificato
%&\leq (C +2\cot\varepsilon_1) \ell_{(x^1)^\bot}(A_1,\ldots, A_r) +  |\tilde{A}_1\tilde{A}_2| + C \sum_{i=1}^{\nu} |A_{q_i}^\bot A_{k_i}^\bot|\\
%&\leq (C +2\cot\varepsilon_1) \ell_{(x^1)^\bot}(A_1,\ldots, A_r) +  |\tilde{A}_1\tilde{A}_2| + C \sum_{i=1}^{\nu}\sum_{j=q_i}^{k_i-1} |A_{j}^\bot A_{j+1}^\bot|\\
&\leq C  \ell_{(x^1)^\bot}(A_1,\ldots, A_r) +  |\tilde{A}_1\tilde{A}_2| + 2\ell_{(x^1)^\bot}(A_1,\ldots, A_r)\cot\varepsilon_1\\
&= (C +2\cot\varepsilon_1) \ell_{(x^1)^\bot}(A_1,\ldots, A_r) +  |\tilde{A}_1\tilde{A}_2|,
\end{align*}
which concludes the proof (up to setting $C_1:=C +2\cot\varepsilon_1$, $C_2:=\diam \B$, where $\B$ is the closed unit ball of $\|\cdot\|$, and recalling that by Lemma~\ref{lm_scontr_horiz3} one can take $C:=\cot\varepsilon_1 +8/\sin\varepsilon_1$), because $|\tilde{A}_1\tilde{A}_2|\leq C_2 |A_1A_r|$ by Remark~\ref{rm_sc1subvec2} (since
$(\tilde{A}_1,\ldots, \tilde{A}_\rho)\subset (A_1,\ldots, A_r)$).
\end{proof}

\begin{corollary}\label{co_scontr_horiz4a}
Under conditions of Lemma~\ref{lm_scontr_base0} one has
\begin{equation}\label{eq_scontr_estv1_2base}
\ell_{x^1}(A_1,\ldots, A_r)\leq C_1  \sum_{i=2}^n \ell_{x^i}(A_1,\ldots, A_r) +  C_2 |A_1A_r|,
\end{equation}
with some positive constants $C_1$ depending only on $\|\cdot\|$, $n$ and $\xi$, and $C_2$  depending only on $\|\cdot\|$.
In particular, one can take
$C_1:= (3\cot\varepsilon_1+ 8/\sin\varepsilon_1)C(\xi)$, where $\varepsilon_1>0$ is defined by Lemma~\ref{lm_scontr_horiz1},
and $C_2:=\diam \B$, $\B$ standing for the closed unit ball of $\|\cdot\|$.
\end{corollary}

\begin{proof}
By Lemma~\ref{lm_scontr_estvarPi-i0} with $i:=2$, one has
\[
\ell_{(x^1)^\bot}(A_1,\ldots, A_r) \leq C(\xi)\sum_{k=2}^n \ell_{x_k}(A_1,\ldots, A_r).
\]
Plugging this estimate into the inequality~\eqref{eq_scontr_estv1_1a} from Lemma~\ref{lm_scontr_horiz4}, one gets the result.
\end{proof}

The following lemmata have been used in the proof of the above Lemma~\ref{lm_scontr_horiz4}.

\begin{lemma}\label{lm_scontr_hordir1}
Let $(A_1,\ldots, A_r)$ be such that
\begin{itemize}
\item[(A)]
each line segment $[A_iA_r]$, $i=1, \ldots, r-1$,
is $\varepsilon$-horizontal with respect to the axis $x$,
passing through $A_r$ parallel to $\nu$, and directed in the direction of $\nu$.
\item[(B)] each line segment $[A_kA_{k+1}]$, $k\in \{1,\ldots, r-1\}$ which is $\varepsilon$-horizontal with respect to $x$ has projection  on $x$ of the same direction as that of $[A_{r-1}A_r]$, i.e.\
denoting by $x_k:= (A_k-A_r)\cdot \nu$, one has
\[
(x_{k+1}-x_k)(x_r-x_{r-1}) \geq 0.
\]
\end{itemize}
Then
$[A_1A_r]$ has projection  on $x$ of the same direction as that of $[A_{r-1}A_r]$, i.e.\
\begin{equation}\label{eq_hodir1_dir1}
(x_r-x_1)(x_r-x_{r-1}) \geq 0,
\end{equation}
 and
\begin{equation}\label{eq_hodir1_var1}
\ell_{x}(A_1,\ldots, A_r)\leq |x_r-x_1| + 2\ell_{x^\bot}(A_1,\ldots, A_r)\cot\varepsilon.
\end{equation}
\end{lemma}

It is worth emphasizing that the above Lemma~\ref{lm_scontr_hordir1} does not require that
$(A_1,\ldots, A_r)$ be self-contracted.

\begin{proof}
We first prove that for each $k=1,\ldots, r-1$, the line segment $[A_kA_r]$ has projection  on $x$ of the same direction as that of $[A_{r-1}A_r]$, i.e.\
\begin{equation}\label{eq_hodir1_dir2}
(x_r-x_k)(x_r-x_{r-1}) \geq 0,
\end{equation}
so that in particular~\eqref{eq_hodir1_dir1} follows. The relationship~\eqref{eq_hodir1_dir2} is proven by backward induction on $k$. In fact, the base $k=r-1$ is automatic, while the inductive step is proven by contradiction as follows. Suppose that~\eqref{eq_hodir1_dir2} holds for some $k=j$, where $j\in \{2,\ldots, r-1\}$, but does not hold for $k=j-1$.
By assumption~(B) this is only possible when $[A_{j-1}A_j]$ is $\varepsilon$-vertical with respect to the $x$ axis.
%We will prove that it contradicts the assumption on $\varepsilon$-horizontality of $[A_{j-1}A_r]$ with respect to the same axis.
Denoting for brevity
$a:= A_r - A_{j-1}$ and $b:= A_j-A_r$, we have that
$a$ and $b$ have the same direction with respect to $x$, i.e.\
$(a\cdot\nu) (b\cdot\nu) >0$ and
$\widehat{(a,\nu')} \leq \varepsilon$, $\widehat{(b,\nu')} \leq \varepsilon$ with either
$\nu'=\nu$ or $\nu'=-\nu$.
Then
\begin{align*}
\cos \widehat{((A_{j-1}A_j),x)} \geq  \cos \widehat{(a+b,\nu')} & =
\frac{\cos(\widehat{(a,\nu')})|a|+ \cos(\widehat{(b,\nu')})|b|}{|a+b|} \\
& \geq \cos \varepsilon \frac{ |a|+|b|}{|a+b|}\geq \cos\varepsilon,
\end{align*}
so that $\widehat{((A_{j-1}A_j),x)}\leq \varepsilon$ contradicting $\varepsilon$-verticality of $[A_{j-1}A_j]$  with respect to the $x$ axis and hence
concluding the proof of~\eqref{eq_hodir1_dir2}.

To prove~\eqref{eq_hodir1_var1}, we note that
\begin{align*}
    x_r-x_1 &=  \sum_{\stackrel{i\in \{1, \ldots, r-1\}}{ [A_iA_{i+1}]\mbox{ horizontal}}} (x_{i+1}-x_i) +
 \sum_{\stackrel{i\in \{1, \ldots, r-1\}}{ [A_iA_{i+1}]\mbox{ vertical}}} (x_{i+1}-x_i),
\end{align*}
so that by~(B), one has
\begin{align*}
    |x_r-x_1| &
\geq   \left|\sum_{\stackrel{i\in \{1, \ldots, r-1\}}{ [A_iA_{i+1}]\mbox{ horizontal}}} (x_{i+1}-x_i)\right| -
 \sum_{\stackrel{i\in \{1, \ldots, r-1\}}{ [A_iA_{i+1}]\mbox{ vertical}}} |x_{i+1}-x_i|\\
&=   \sum_{\stackrel{i\in \{1, \ldots, r-1\}}{ [A_iA_{i+1}]\mbox{ horizontal}}} |x_{i+1}-x_i| -
 \sum_{\stackrel{i\in \{1, \ldots, r-1\}}{ [A_iA_{i+1}]\mbox{ vertical}}} |x_{i+1}-x_i|\\
& =  \sum_{i=1}^{r-1} |x_{i+1}-x_i| -
 2\sum_{\stackrel{i\in \{1, \ldots, r-1\}}{ [A_iA_{i+1}]\mbox{ vertical}}} |x_{i+1}-x_i|\\
&\geq \sum_{i=1}^{r-1} |x_{i+1}-x_i| -
 2\sum_{\stackrel{i\in \{1, \ldots, r-1\}}{ [A_iA_{i+1}]\mbox{ vertical}}} |p_{x^\bot}(A_{i+1}-A_i)|\cot\varepsilon\\
&\geq \ell_{x}(A_1,\ldots, A_r)-2\ell_{x^\bot}(A_1,\ldots, A_r)\cot\varepsilon,
\end{align*}
concluding the proof.
\end{proof}

\begin{lemma}\label{lm_scontr_horiz3}
Assume that the vector $(A_1,\ldots, A_r)$ %with $A_j\in \P_{\alpha_i}$ for all $j=1,\ldots, r$,
be self-contracted with respect to
the norm $\|\cdot\|$. Let also $\varepsilon\in (0,\varepsilon_1]$, where $\varepsilon_1$ is defined by Lemma~\ref{lm_scontr_horiz1},
%and Lemma~\ref{lm_scontr_verthoriz1}
%holds and
%$\delta<\varepsilon$,
and for each line segment $[A_kA_{k+1}]$ with $k=2,\ldots, r-1$ %of $(A_1,\ldots, A_r)$
which is $\varepsilon$-horizontal with respect to $x^1$ axis the preceding line segment $[A_{k-1}A_k]$ is  either also
$\varepsilon$-horizontal with respect to $x^1$ axis and its projection
on $x^1$ is directed oppositely to that of $[A_kA_{k+1}]$, i.e.\
\[
(x^1_{k+1}-x^1_k)(x^1_k-x^1_{k-1}) < 0,
\]
or is $\varepsilon$-vertical with respect to $x^1$ axis, and so is $[A_{k-1}A_{k+1}]$.
Then
\begin{equation}\label{eq_scontr_estv1_0}
\ell_{x^1}(A_1,\ldots, A_r)\leq C \ell_{(x^1)^\bot}(A_1,\ldots, A_r) +  |A_1A_2|,
\end{equation}
with some positive constant $C$ %and $C_2$
depending only %on the norm $\|\cdot\|$ and
on $\varepsilon$ (and in particular, one can take $C:=\cot\varepsilon+ 8/\sin\varepsilon$).
\end{lemma}

\begin{proof}
We denote by $H$ (resp.\ $V$) the set of subvectors $(A_{q_k}, A_{q_k+1},\ldots, A_{q_{k+1}})\subset (A_1,\ldots, A_r)$
such that each line segment $[A_iA_{i+1}]$, $i=q_k, \ldots, q_{k+1}-1$ is $\varepsilon$-horizontal (resp.\ $\varepsilon$-vertical) with respect to $x^1$ axis
and either $q_k=1$ or the preceding line segment line $[A_{q_k-1}A_{q_k}]$ is $\varepsilon$-vertical (resp.\ $\varepsilon$-horizontal) with respect to $x^1$ axis.
Consider the partition $1=q_1<q_2<\ldots< q_\sigma=r$ of the set $\{1,\ldots, r\}$
such that each subvector  $(A_{q_k}, A_{q_k+1},\ldots, A_{q_{k+1}})$, $k=1,\ldots, \sigma$,
either belongs to $H$ or to $V$.

For $(A_{q_k}, A_{q_k+1},\ldots, A_{q_{k+1}})\in H$  by
the assumption of the statement being proven the line segments of this polygonal line have alternating directions
with respect to $x^1$ axis
and hence by Corollary~\ref{co_scontr_horiz2} we have
\[
\ell_{x^1}\left(\left\{A_j\right\}_{j={q_k}}^{q_{k+1}}\right)\leq
\ell\left(\left\{A_j\right\}_{j={q_k}}^{q_{k+1}}\right)\leq
%5 |A_{q_k}^1A_{q_k+1}^1|.
4 |A_{q_k}A_{q_k+1}|.
\]
For $(A_{q_k}, A_{q_k+1},\ldots, A_{q_{k+1}})\in V$, we just estimate
\begin{align*}
\ell_{x^1}\left(\left\{A_j\right\}_{j={q_k}}^{q_{k+1}}\right)  &= \sum_{j=q_k}^{q_{k+1}-1} |A_j^1A_{j+1}^1|
%\\
%&
\leq \sum_{j=q_k}^{q_{k+1}-1} |A_j^\bot A_{j+1}^\bot|\cot\varepsilon = \ell_{(x^1)^\bot}\left(\left\{A_j\right\}_{j={q_k}}^{q_{k+1}}\right)\cot\varepsilon.
\end{align*}
We have now
\begin{equation}\label{eq_scontr1_estv1}
\begin{aligned}
\ell_{x^1} (A_1,\ldots, A_r)  &= \sum_{k=1}^{\sigma-1} \ell_{x^1}(A_{q_k}, \ldots, A_{q_{k+1}})\\
 & = \sum_{\stackrel{k\in \{1,\ldots,\sigma-1\},}{\left\{A_j\right\}_{j={q_k}}^{q_{k+1}}\in V}} \ell_{x^1}\left(\left\{A_j\right\}_{j={q_k}}^{q_{k+1}}\right)
+\sum_{\stackrel{k\in \{1,\ldots,\sigma-1\},}{\left\{A_j\right\}_{j={q_k}}^{q_{k+1}}\in H}} \ell_{x^1}\left(\left\{A_j\right\}_{j={q_k}}^{q_{k+1}}\right)\\
&\leq \cot\varepsilon \sum_{\stackrel{k\in \{1,\ldots,\sigma-1\},}{\left\{A_j\right\}_{j={q_k}}^{q_{k+1}}\in V}} \ell_{(x^1)^\bot}\left(\left\{A_j\right\}_{j={q_k}}^{q_{k+1}}\right) +
4 \sum_{\stackrel{k\in \{1,\ldots,\sigma-1\},}{\left\{A_j\right\}_{j={q_k}}^{q_{k+1}}\in H}} |A_{q_k}A_{q_k+1}|.
%5 \sum_{\stackrel{k\in \{1,\ldots,\sigma-1\},}{\left\{A_j\right\}_{j={q_k}}^{q_{k+1}}\in H}} |A_{q_k}^1A_{q_k+1}^1|. %% ok ma semplificato
\end{aligned}
\end{equation}
For each $k\in \{1,\ldots,\sigma-1\}$ such that $(A_{q_k},\ldots, A_{q_{k+1}})\in H$ except $k=1$
we estimate %by Lemma~\ref{lm_scontr_verthoriz1}
\begin{align*}
|A_{q_k}A_{q_k+1}| & %\leq |A_{q_k}A_{q_k+1}| \leq |A_{q_k-1}A_{q_k}| + |A_{q_k-1}A_{q_k+1}|\\ %%% ok ma abbreviato
     \leq |A_{q_k-1}A_{q_k}| + |A_{q_k-1}A_{q_k+1}|
%\\
%&
\leq \frac{1}{\sin\varepsilon} \left(|A_{q_k-1}^\bot A_{q_k}^\bot| + |A_{q_k-1}^\bot A_{q_k+1}^\bot|\right)\\
&\leq \frac{1}{\sin\varepsilon} \left(2|A_{q_k-1}^\bot A_{q_k}^\bot| + |A_{q_k}^\bot A_{q_k+1}^\bot|\right)
%\\
%&
\leq
\frac{2}{\sin\varepsilon}\left(|A_{q_k-1}^\bot A_{q_k}^\bot|+|A_{q_k}^\bot A_{q_k+1}^\bot|\right),
\end{align*}
and if
$(A_{q_1},\ldots, A_{q_{2}})=(A_1,\ldots, A_{q_{2}})\in H$, then just
\[
|A_{q_1}A_{q_1+1}| =  |A_1A_2|.
\]
%\[ %%% ok ma semplificato
%|A_{q_1}^1A_{q_1+1}^1| = |A_1^1A_2^1|\leq |A_1A_2|.
%\]
Plugging this into~\eqref{eq_scontr1_estv1}, we get
\[
%\begin{equation}\label{eq_scontr1_estv1a}
\begin{aligned}
\ell_{x^1}& \left(\left\{A_j\right\}_{j=1}^r\right)  \leq \cot\varepsilon \sum_{\stackrel{k\in \{1,\ldots,\sigma-1\},}{\left\{A_j\right\}_{j={q_k}}^{q_{k+1}}\in V}} \sum_{j={q_k}}^{q_{k+1}-1} |A_{j+1}^\bot A_{j}^\bot| +
\\
&\qquad\qquad\qquad\qquad
\frac{8}{\sin\varepsilon} \sum_{\stackrel{k\in \{2,\ldots,\sigma-1\},}{\left\{A_j\right\}_{j={q_k}}^{q_{k+1}}\in H}} \left(|A_{q_k-1}^\bot A_{q_k}^\bot|+|A_{q_k}^\bot A_{q_k+1}^\bot|\right) +  |A_1A_2|\\
&\leq \left(\cot\varepsilon+ \frac{8}{\sin\varepsilon}\right)\sum_{\stackrel{k\in \{1,\ldots,\sigma-1\},}{\left\{A_j\right\}_{j={q_k}}^{q_{k+1}}\in V}} \sum_{j={q_k}}^{q_{k+1}-1} |A_{j+1}^\bot A_{j}^\bot| +
\\
&\qquad\qquad\qquad\qquad
\frac{8}{\sin\varepsilon}  \sum_{\stackrel{k\in \{2,\ldots,\sigma-1\},}{\left\{A_j\right\}_{j={q_k}}^{q_{k+1}}\in H}} |A_{q_k}^\bot A_{q_k+1}^\bot| +  |A_1A_2|,
%\\
%&\leq \left(\cot\varepsilon+ \frac{8}{\sin\varepsilon}\right) \ell_{(x^1)^\bot} (A_1,\ldots, A_r)+|A_1A_2|
\end{aligned}
%\end{equation}
\]
because $[A_{q_k-1} A_{q_k}]$ is $\varepsilon$-vertical with respect to $x^1$ axis for $\left\{A_j\right\}_{j={q_k}}^{q_{k+1}}\in H$. Hence
\[
\ell_{x^1} \left(\left\{A_j\right\}_{j=1}^r\right)  \leq \left(\cot\varepsilon+ \frac{8}{\sin\varepsilon}\right) \ell_{(x^1)^\bot} (A_1,\ldots, A_r)+|A_1A_2|
\]
as claimed,
concluding the proof.
\end{proof}

\subsubsection{Estimate of $\ell(A_1,\ldots, A_r)$}

Finally we are able to prove Lemma~\ref{lm_scontr_base0} providing the estimate on the total variation of $(A_1,\ldots, A_r)$.

\begin{proof}[Proof of Lemma~\ref{lm_scontr_base0}]
Plugging~\eqref{eq_scontr_estv1_2base} (Corollary~\ref{co_scontr_horiz4a}) into~\eqref{eq_scontr_estv1_base0j}
(Lemma~\ref{lm_scontr_vert1}),
we get
\[
\ell_{x^j}(A_1,\ldots, A_r)\leq (2+4C_2\tan\delta) |A_1A_r| +4C_1\tan\delta \sum_{i=2}^n \ell_{x^i}(A_1,\ldots, A_r),\quad
j=2, \ldots, r,
\]
where $C_1$ and $C_2$ are as in Corollary~\ref{co_scontr_horiz4a}, i.e.\
$C_1:= (3\cot\varepsilon_1+ 8/\sin\varepsilon_1)C(\xi)$
and $C_2:=\diam \B$, where $\B$ is the closed unit ball of $\|\cdot\|$.
We may now apply Lemma~\ref{lm_scontr_sumLj1} with
\begin{align*}
L_0& :=(2+4C_2\tan\delta)|A_1A_r|,\\
L_j &:=\ell_{x_j}(A_1,\ldots, A_r),
\end{align*}
 $i:=2$ and $C:=4C_1\tan\delta$ (recalling $\tan\delta <1/(8(n-1)C_1)$ under conditions of the statement being proven)
to get
\[
\ell_{x^j}(A_1,\ldots, A_r)\leq (4+8C_2\tan\delta) |A_1A_r|, \quad j=2, \ldots, r,
\]
and plugging the latter estimate back into~\eqref{eq_scontr_estv1_2base}, we get
\[
\ell_{x^1}(A_1,\ldots, A_r)\leq (C_1(n-1)(4+8C_2\tan\delta)+C_2) |A_1A_r|,
\]
which concludes the proof.
\end{proof}

\subsection{Inductive step}

We prove now the following statement that will serve as an inductive step.

\begin{lemma}\label{lm_scontr_ind0}
Assume that %$\delta\in (0, \delta_0)$, where
\[
%\tan\delta_0\leq \frac{\sin\varepsilon}{4 n %(n-2)
%C(\xi)},
\delta\in \left(0, \bar\delta\wedge\arctan \frac{\sin\varepsilon_0}{8 (n-1)
C(\xi)}\right),
\]
where $\varepsilon_0>0$ is defined in Lemma~\ref{lm_scontr_defxi1},
and suppose that the following inductive hypothesis holds:
if for some $i\in \{2, \ldots, n-1\}$
there exists a constant $C_{i-1}>0$ such that
whenever the
$(n-i+2)$-tuple $(\P_{\alpha_{i-1}},\ldots, \P_{\alpha_n})$ is admissible,
the vector
$(A_1,\ldots, A_r, A_{r+1}, \ldots A_{r+n-i+1})$
is self-contracted with respect to
the norm $\|\cdot\|$ and
$(A_1,\ldots, A_r)\subset \P_{\alpha_j} +A_{r+n-j}$
for all $j=i-1,\ldots, n$,  then
\begin{equation}\label{eq_scontr_estv1_ind1hyp}
\ell(A_1,\ldots, A_r)\leq C_{i-1}|A_1A_r|.
\end{equation}
Then this property holds also for $i+1$, i.e.\
there exists a constant $C_{i}>0$ (depending only on the norm $\|\cdot\|$, on $C_{i-1}$, on $n$, $\delta$ and $\xi$) such that
whenever the $(n-i+1)$-tuple $(\P_{\alpha_i},\ldots, \P_{\alpha_n})$ is admissible, the vector
$(A_1,\ldots, A_r, A_{r+1}, \ldots, A_{r+n-i})$
is self-contracted with respect to
the norm $\|\cdot\|$ and
$(A_1,\ldots, A_r)\subset \P_{\alpha_j} +A_{r+n-j}$
for all $j=i,\ldots, n$,  then
\begin{equation}\label{eq_scontr_estv1_ind1}
\ell(A_1,\ldots, A_r)\leq C_{i}|A_1A_r|.
\end{equation}
\end{lemma}

%The rest of this section will be dedicated to the proof of the above Lemma~\ref{lm_scontr_ind0}.

\begin{proof}%[Proof of Lemma~\ref{lm_scontr_ind0}]
Let %$\P_{\alpha_i}$ be a pyramid in an admissible
the $n$-tuple of sets $(\P_{\alpha_i},\ldots, \P_{\alpha_n})$
with
$\alpha_j\in \{1,\ldots, N\}$, $j=i, \ldots, n$,
be admissible, the vector
$(A_j)_{j=1}^{r+n-i}$
%$(A_1,\ldots, A_r, A_{r+1}, \ldots A_{r+n-1})$
be self-contracted with respect to
the norm $\|\cdot\|$ and
$(A_1,\ldots, A_r)\subset \P_{\alpha_j} +A_{r+n-j}$
for all $j=i,\ldots, n$. Let $x^j$ be the axis in the direction $\nu^{\alpha_j}$ determined by $\P_{\alpha_j}$ passing through
$A_r$, and consider the $(i-1)$-dimensional linear subspace
$\Pi^{i-1}:= (\mathrm{span}\, \{\nu^{\alpha_j}\}_{j=i}^n)^\bot$.
%where  $(x^i, x^{i+1}, x^n)$ stands, as usual, for the
%$(n-i+1)$-dimensional hyperplane determined by the axes $x^j$ for $j=i, \ldots, n$.
The rest of the proof will be organized in several steps.

{\sc Step 1}.
We act similarly to the proof of Lemma~\ref{lm_scontr_horiz4}. Namely,
starting from the last  line segment $[A_{k-1}A_k]$ with $1<k\leq r$ which is $\varepsilon_0$-horizontal with respect to $\Pi^{i-1}$ (i.e.\ with maximum $k\in  \{2,\ldots, r\}$
such that $[A_{k-1}A_k]$ is $\varepsilon_0$-horizontal with respect to this subspace), if it exists,
we find the minimum $q\in \{1, \ldots, k-1\}$ such that for every $j\in \{q, \ldots, k-1\}$ the line segment
$[A_jA_k]$ is $\varepsilon_0$-horizontal with respect to $\Pi^{i-1}$. Note that in this way, if $q>1$, then $[A_{q-1}A_k]$ is $\varepsilon_0$-vertical with respect to $\Pi^{i-1}$.
Now we repeat the same operation
with  $q-1$
%\textbf{YT: $q$ or $q-1$? Here I guess $q-1$!}
instead of $r$, and continue in this way by backward induction as far as possible. In this way we find a
finite (possibly empty) sequence of disjoint subintervals $\{q_l,q_l+1,\ldots, k_l\}_{l=1}^{\nu}$ of $\{1,\ldots, r\}$
such that for every $j\in \{q_l, \ldots, k_l-1\}$ the line segment
$[A_jA_{k_l}]$ is $\varepsilon_0$-horizontal with respect to $\Pi^{i-1}$, while
$[A_{q_l-1}A_{k_l}]$ is $\varepsilon_0$-vertical
with respect to this subspace (if $q_l>1$).

We claim that
\begin{equation}\label{eq_scontr_varAikjk1}
\ell(A_{q_l},\ldots, A_{k_l})\leq \bar C_1 \sum_{\stackrel{\P_{\alpha_{i-1}}\colon (\P_{\alpha_{i-1}},\P_{\alpha_i},\ldots, \P_{\alpha_n})}{\mathrm{admissible}}}
\ell((A_{q_l},\ldots, A_{k_l})\cap (\P_{\alpha_{i-1}}+A_{k_l})) + \bar C_2  |A_{q_l}A_{k_l}|
\end{equation}
for all $l=1,\ldots,\nu$,
where $\bar C_1>0$ depends only on the norm $\|\cdot\|$, and $\bar C_2>0$ depends on the norm and on $\delta$.
To show this claim, consider an arbitrary $j\in \{q_l,\ldots, k_l-1\}$ such that $A_{j+1}\in \P_{\alpha_{i-1}}+A_{k_l}$, but $A_j\not \in \P_{\alpha_{i-1}}+A_{k_l}$
for some $\P_{\alpha_{i-1}}$ such that $(\P_{\alpha_{i-1}},\P_{\alpha_i},\ldots, \P_{\alpha_n})$ is admissible. Then either
\begin{itemize}
  \item[(i)] $j$ is the first index in $\{q_l,\ldots, k_l\}$ such that $A_{j+1}\in \P_{\alpha_{i-1}}+A_{k_l}$, i.e.\
  \[
  \{ s\in \{q_l,\ldots, j\}\colon A_s\in \P_{\alpha_{i-1}}+A_{k_l}\}=\emptyset,
  \]
in which case
 we just use $(A_j,A_{j+1})\subset (A_{q_l},\ldots, A_{k_l})$ to estimate $|A_jA_{j+1}|\leq C|A_{q_l}A_{k_l}|$
by Remark~\ref{rm_sc1subvec2}
for some $C>0$ depending only on $\|\cdot\|$; the sum of Euclidean lengths of all such line segments $|A_jA_{j+1}|$ through all $\P_{\alpha_{i-1}}$ such that $(\P_{\alpha_{i-1}},\P_{\alpha_i},\ldots, \P_{\alpha_n})$ is admissible, is estimated therefore from above by $\bar C_2 |A_{q_l}A_{k_l}|$, where $\bar C_2 := C N(\delta)$;
  \item[(ii)] or there is an
\[
s(j):=\max\{ s\in \{q_l,\ldots, j\}\colon A_s\in \P_{\alpha_{i-1}}+A_{k_l}\},
\]
and $s(j)<j$ by the definition of $s(\cdot)$, hence
$(A_j,A_{j+1})\subset (A_{s(j)},\ldots, A_{j+1})$ which implies $|A_jA_{j+1}|\leq C|A_{s(j)}A_{j+1}|$
again by Remark~\ref{rm_sc1subvec2}
for $C>0$ depending only on $\|\cdot\|$ (same as $C$ in~(i)). Therefore, with $\bar C_1:=C\vee 1$ one has
\begin{align*}
\sum_{\stackrel{j\in \{q_l,\ldots, k_l-1\}}{\{A_j,A_{j+1}\}\subset \P_{\alpha_{i-1}}+A_{k_l}}}  & |A_jA_{j+1}| +
\sum_{\stackrel{j\in \{q_l,\ldots, k_l-1\}}{\mathrm{as~in~(ii)}}} |A_{j}A_{j+1}| \\
&\leq
\sum_{\stackrel{j\in \{q_l,\ldots, k_l-1\}}{\{A_j,A_{j+1}\}\subset \P_{\alpha_{i-1}}+A_{k_l}}} |A_jA_{j+1}| +
C\sum_{\stackrel{j\in \{q_l,\ldots, k_l-1\}}{\mathrm{as~in~(ii)}}} |A_{s(j)}A_{j+1}| \\
&\leq \bar C_1 \left(\sum_{\stackrel{j\in \{q_l,\ldots, k_l-1\}}{\{A_j,A_{j+1}\}\subset \P_{\alpha_{i-1}}+A_{k_l}}} |A_jA_{j+1}| +
\sum_{\stackrel{j\in \{q_l,\ldots, k_l-1\}}{\mathrm{as~in~(ii)}}} |A_{s(j)}A_{j+1}| \right)\\
&\quad\quad = \bar C_1 \ell((A_{q_l},\ldots, A_{k_l})\cap (\P_{\alpha_{i-1}}+A_{k_l})).
\end{align*}
\end{itemize}
From~(i) and~(ii) we get therefore~\eqref{eq_scontr_varAikjk1}.

{\sc Step 2}. By the inductive assumption for each $\P_{\alpha_{i-1}}$ with
$(\P_{\alpha_{i-1}},\P_{\alpha_i},\ldots, \P_{\alpha_n})$ admissible one has
for
$(A_j)_{j\in \Lambda_{k_l}}:=
(A_{q_l},\ldots, A_{k_l})\cap (\P_{\alpha_{i-1}}+A_{k_i})$ (clearly, $\Lambda_{k_l}\subset \{q_l, q_l+1, \ldots, k_l\}$, and
$k_l\in \Lambda_{k_l}$) the estimate
\begin{equation}\label{eq_scontr_varAikjk2}
\ell((A_j)_{j\in \Lambda_{k_l}})\leq C_{i-1} |A_{\tilde q_l} A_{k_l}|,
\end{equation}
where $\tilde q_l$ stands for the first index in $\Lambda_{k_l}$. But since $(A_j)_{j\in \Lambda_{k_l}}\subset
(A_{q_l},\ldots, A_{k_l})$, then
\[
|A_{\tilde q_l} A_{k_l}|\leq C |A_{q_l} A_{k_l}|
\]
for some $C>0$ depending only on the norm $\|\cdot\|$, and hence~\eqref{eq_scontr_varAikjk2} implies
$
%\begin{equation}\label{eq_scontr_varAikjk3}
\ell((A_j)_{j\in \Lambda_{k_l}})\leq C |A_{q_l} A_{k_l}|.
%\end{equation}
$
Therefore, from~\eqref{eq_scontr_varAikjk1} we get
\begin{equation}\label{eq_scontr_varAikjk4}
\ell(A_{q_l},\ldots, A_{k_l})\leq C' |A_{q_l}A_{k_l}|
\end{equation}
for all $l=1,\ldots,\nu$, where $C':=\bar C_1 C_{i-1} C N(\delta)  + \bar C_2>0$ depends on $\delta$ and on $\|\cdot\|$, as well as on $C_{i-1}$.

{\sc Step 3}.
Consider the subvector
$
%\begin{align*}
(A_j)_{j\in \Lambda}
%&:=( A_1, \ldots, A_{q_1-1}, A_{q_1}, A_{k_1}, \ldots, A_{q_2-1},A_{q_2}, A_{k_2},\ldots, A_{q_\nu-1},A_{q_\nu}, A_{k_\nu}, A_r)
%\\
%&
\subset (A_1,\ldots, A_r)
%\end{align*}
$
obtained from $(A_1,\ldots, A_r)$ by canceling all the points $A_j$ with $j\in \cup_{i=1}^{\nu}\{q_i+1,\ldots, k_i-1\}$.
The inequality~\eqref{eq_scontr_varAikjk4} yields then, observing that clearly $C'>1$, the estimate
\begin{equation}\label{eq_scontr_varAikjk5}
\ell(A_1,\ldots, A_r)\leq C' \ell((A_j)_{j\in \Lambda}).
\end{equation}
To estimate the right-hand side of~\eqref{eq_scontr_varAikjk5}, we first note that by Lemma~\ref{lm_sc1revtriangle}
\[
|A_{q_i-1}A_{q_i}|+|A_{q_i}A_{k_i}|\leq C|A_{q_i-1}A_{k_i}|
\]
for a $C>0$ depending only on $\|\cdot\|$, if $q_i>1$, and thus for
a subvector $(A_j)_{j\in \tilde \Lambda}\subset (A_j)_{j\in \Lambda}$ obtained from
$(A_j)_{j\in \Lambda}$ by canceling all $A_{q_i}$, $i=1,\ldots, \nu$ except possibly $q_i=\min\Lambda=1$ (i.e.\
$q_i$ equal to
the first index in $\Lambda$ which is $1$ by construction), and recalling
that $\{q_i-1, q_i, k_i\}\subset \Lambda$ when $q_i>1$ again by construction, we get
\begin{equation}\label{eq_scontr_varAikjk6}
\ell((A_j)_{j\in \Lambda})\leq C \ell((A_j)_{j\in \tilde \Lambda}).
\end{equation}
Thus, in view of~\eqref{eq_scontr_varAikjk6} and~\eqref{eq_scontr_varAikjk5} the proof will be concluded
once we show that
for some $\tilde C >0$ depending possibly on %$\delta$, %$\varepsilon_0$
$\xi$, hence on $\|\cdot\|$, as well as on $n$ and $i$, one has
\begin{equation}\label{eq_scontr_varAikjk6a}
\ell((A_j)_{j\in \tilde \Lambda})\leq \tilde C |A_{\min \tilde \Lambda} A_{\max \tilde \Lambda}|=\tilde C |A_1 A_r|,
\end{equation}
the last equality being due to the fact that  $\min \tilde \Lambda=1$ and $\max\tilde \Lambda=r$ by construction.

{\sc Step 4}. It remains to prove~\eqref{eq_scontr_varAikjk6a}.
We will in fact show it with $\tilde C >0$ depending possibly on $\xi$, hence just on %$\varepsilon_0$ and
$\|\cdot\|$.
To this aim observe that
all the segments of the polygonal line $(A_j)_{j\in \tilde\Lambda}$ except possibly the first one are $\varepsilon_0$-vertical with respect to
$\Pi^{i-1}$.
Denoting for convenience %$\rho:=\max \tilde\Lambda$ (i.e.\
$(\tilde A_1, \ldots, \tilde A_\rho):=(A_j)_{j\in \tilde \Lambda}$,
%),
we
have therefore that
\begin{equation}\label{eq_scontr_varAikjk7}
\begin{aligned}
\ell_{\Pi^{i-1}}((A_j)_{j\in \tilde \Lambda})&\leq \ell((A_j)_{j\in \tilde \Lambda})
  = |\tilde A_1\tilde A_2| + \ell(\tilde A_2, \ldots, \tilde A_\rho)\\
& \leq |\tilde A_1\tilde A_2| + \frac{1}{\sin\varepsilon_0}\ell_{(\Pi^{i-1})^\bot} ((A_j)_{j\in \tilde \Lambda})\\
&\leq C |\tilde A_1\tilde A_\rho| + \frac{1}{\sin\varepsilon_0}\ell_{(\Pi^{i-1})^\bot} ((A_j)_{j\in \tilde \Lambda})\quad
\mbox{by Remark~\ref{rm_sc1subvec2},}
\end{aligned}
\end{equation}
with a constant $C>0$ depending only on $\|\cdot\|$ (here in the last inequality we used
$(\tilde A_1,\tilde A_2)\subset (A_j)_{j\in \tilde \Lambda}$). But for each $j\in \{i,\ldots, n\}$ one has
\begin{equation}\label{eq_scontr_varAikjk8}
\begin{aligned}
\ell_{x^j}((A_j)_{j\in \tilde \Lambda}) &\leq |\tilde A_1\tilde A_\rho| + 2\tan\delta
 \ell_{(x^j)^\bot }((A_j)_{j\in \tilde \Lambda})\quad\mbox{by Lemma~\ref{lm_scontr_vert1a}}\\
& \leq |\tilde A_1\tilde A_\rho| + 2\tan\delta
 \ell((A_j)_{j\in \tilde \Lambda})\\
&\leq |\tilde A_1\tilde A_\rho| + 2\tan\delta\left(
 \ell_{\Pi^{i-1}}((A_j)_{j\in \tilde \Lambda})
%\\
%&\qquad\qquad
+ \ell_{(\Pi^{i-1})^\bot}
((A_j)_{j\in \tilde \Lambda})
\right) \\
&\leq |\tilde A_1\tilde A_\rho| + 2\tan\delta
 \ell_{\Pi^{i-1}}((A_j)_{j\in \tilde \Lambda})
%\\
%& \qquad\qquad
+
2 C(\xi) \tan\delta\sum_{k=i}^n  \ell_{x^k}((A_j)_{j\in \tilde \Lambda}), %\quad\mbox{by Lemma~\ref{lm_scontr_estvarPi-i0}}.
\end{aligned}
\end{equation}
the latter inequality being due to Lemma~\ref{lm_scontr_estvarPi-i0}.
By Lemma~\ref{lm_scontr_sumLj1} with
\begin{align*}
L_0 &:= |\tilde A_1\tilde A_\rho| + 2\tan\delta
 \ell_{\Pi^{i-1}}((A_j)_{j\in \tilde \Lambda}),\\
L_k &:=\ell_{x^k}((A_j)_{j\in \tilde \Lambda})
\end{align*}
and $C:= 2 C(\xi) \tan\delta$,
%
%and summing the above inequalities over $j=i,\ldots, n$, we get
%\[
%%\begin{equation}\label{eq_scontr_varAikjk9}
%\begin{aligned}
%\sum_{j=i}^n \ell_{x^j}((A_j)_{j\in \tilde \Lambda})
%&\leq
%(n-i+1) |\tilde A_1\tilde A_\rho| + (n-i+1)\tan\delta
% \ell_{\Pi^{i-1}}((A_j)_{j\in \tilde \Lambda}) + \\
%&\qquad\qquad 2(n-i) C(\xi) \tan\delta \sum_{j=i}^n \ell_{x^j}((A_j)_{j\in \tilde \Lambda}),
%\end{aligned}
%%\end{equation}
%\]
%which,
recalling that $\tan\delta < 1/(4(n-i+1) C(\xi))$ by assumption of the statement being proven, we obtain then from~\eqref{eq_scontr_varAikjk8} the estimate
\begin{equation}\label{eq_scontr_varAikjk9}
\begin{aligned}
\sum_{j=i}^n \ell_{x^j}((A_j)_{j\in \tilde \Lambda}) & \leq
2 (n-i+1) |\tilde A_1\tilde A_\rho| +
4 (n-i+1) \tan\delta
\ell_{\Pi^{i-1}}((A_j)_{j\in \tilde \Lambda}).
\end{aligned}
\end{equation}
%Applying Lemma~\ref{lm_scontr_estproj1} with $\Pi:=(\Pi^{i-1})^\bot$, $\nu_j:=\nu^{\alpha_j}$, $j=i, \ldots, n$, $k:=n-i$,
By Lemma~\ref{lm_scontr_estvarPi-i0}
we have
\begin{equation}\label{eq_scontr_varAikjk10}
\begin{aligned}
\ell_{(\Pi^{i-1})^\bot }((A_j)_{j\in \tilde \Lambda}) &\leq
C(\xi) \sum_{j=i}^n \ell_{x^j}((A_j)_{j\in \tilde \Lambda})\\
\\
&\leq
2 (n-i+1) C(\xi) |\tilde A_1\tilde A_\rho| +
4 (n-i+1) C(\xi) \tan\delta \ell_{\Pi^{i-1}}((A_j)_{j\in \tilde \Lambda}),%\quad\mbox{by~\eqref{eq_scontr_varAikjk9}}.
\end{aligned}
\end{equation}
the latter inequality being due to~\eqref{eq_scontr_varAikjk9}.
Plugging~\eqref{eq_scontr_varAikjk10} into~\eqref{eq_scontr_varAikjk7} yields
\begin{align*}
\ell_{\Pi^{i-1}}((A_j)_{j\in \tilde \Lambda}) &\leq
\left( C+ \frac{2 (n-i+1) C(\xi)}{\sin\varepsilon_0} \right) |\tilde A_1\tilde A_\rho| +
\frac{4(n-i+1)}{\sin\varepsilon_0} C(\xi)\tan\delta \ell_{\Pi^{i-1}}((A_j)_{j\in \tilde \Lambda}),
\end{align*}
and hence, since $\tan\delta < \sin\varepsilon_0/(8(n-i+1)C(\xi))$, one has
\begin{equation}\label{eq_scontr_varAikjk11}
\begin{aligned}
\ell_{\Pi^{i-1}}((A_j)_{j\in \tilde \Lambda}) &\leq
2\left( C+ \frac{2 (n-i+1) C(\xi)}{\sin\varepsilon_0} \right)
 |\tilde A_1\tilde A_\rho|.
\end{aligned}
\end{equation}
Finally, plugging~\eqref{eq_scontr_varAikjk11} into~\eqref{eq_scontr_varAikjk10}
and recalling
\[
4 (n-i+1) C(\xi) \tan\delta\leq \frac 1 2\sin \varepsilon_0, %\leq \frac 1 2,
\]
we get
\[
\ell_{(\Pi^{i-1})^\bot}((A_j)_{j\in \tilde \Lambda}) \leq
\left(4 (n-i+1) C(\xi)+ C\sin\varepsilon_0\right)|\tilde A_1\tilde A_\rho|,
\]
and which together with~\eqref{eq_scontr_varAikjk11}
gives~\eqref{eq_scontr_varAikjk6a} with $\tilde C:=3 C+ 4 (n-i+1) C(\xi)\left(1+ 1/\sin\varepsilon_0\right)$ as claimed.
%\begin{equation}\label{eq_scontr_varAikjk12} %% ok ma abbreviato
%\ell((A_j)_{j\in \tilde \Lambda}) \leq
%\left(3 C+ 4 (n-i+1) C(\xi)\left(\frac{1}{\sin\varepsilon_0}+1 \right)\right)
%|\tilde A_1\tilde A_\rho|,
%\end{equation}
%that is in fact~\eqref{eq_scontr_varAikjk6a} as claimed.
\end{proof}

\appendix

\section{Auxiliary lemmata}

We used the following easy statements.

%%%%% ok ma non serve
%\begin{lemma}\label{lm_sc1endpt1}
%Let $\theta\colon [0,+\infty)\to K$ be a continuous self-contracted curve in any compact metric space $K$. Then there is
%an $y=\lim_{t\to +\infty} \theta(t)$.
%\end{lemma}
%
%\begin{proof}
%\end{proof}

\begin{lemma}\label{lm_sc1revtriangle}
If $(A_1,A_2,A_3)$ is self-contracted with respect to the norm $\|\cdot\|$, then
\[
|A_1A_2|+|A_2A_3|\leq C |A_1 A_3|,
\]
for some $C>0$ depending only on $\|\cdot\|$.
\end{lemma}

\begin{proof}
In fact, $|A_2A_3|\leq C |A_1 A_3|$ with $C$ as in the statement, which together with the triangle inequality
\[
|A_1A_2|\leq |A_1A_3|+|A_2A_3|,
\]
implies the claim.
\end{proof}

For every couple of distinct points $\{A, B\}\subset \R^n$ denote by
$M(A,B)$ the closed set
\[
M(A,B):=\{z\in \R^n\colon \|A-z\| =\|B-z\|\}
\]
 (called \emph{mediatrix} or \emph{equidistant set} of $\{A, B\}$).
Clearly,  $M(\lambda A, \lambda B)=\lambda M(A, B)$ for all $\lambda\geq 0$ (in fact, even
for all $\lambda\in \R$ if the norm is, as it is customary to assume, symmetric).
Further, $M(A+x, B+x)= M(A,B)$.
Observe also that $M(A,B)\cap (A B)$ is the
midpoint $C_0$ of $[A B]$ if the norm is symmetric.

\begin{lemma}\label{lm_scontr_mediatr1}
There is a constant $\bar{\varepsilon}>0$ depending only on the norm $\|\cdot\|$ such that
for every $C\in M(A,B)$ with $\angle A B C \leq 2\bar{\varepsilon}$ one has
\[
|BC|\leq 3/4 |AB| .
\]
\end{lemma}

\begin{proof}
Since the statement is invariant with respect to translation and scaling, we may assume without loss of generality that
$|AB|=1$ and $B$ is the origin,
so that $M(A,B)$ depends
only on the direction $\nu\in S^{n-1}$ of the segment $[BA]$.
Let $\varepsilon(\nu)$ stand for the maximum angle
$\alpha$ such that
for
$C\in M(A,B)$ with $\angle ABC \leq \alpha$
 one has $|BC|\leq 3/4$.
It suffices to set now $\bar{\varepsilon}:=\inf_{\nu\in S^{n-1}}\varepsilon(\nu)$ and observe that $\bar{\varepsilon}>0$.
In fact, otherwise there is a sequence $\{A_k\}\subset S^{n-1}$, $A_k\to A$, and $\{C_k\}\subset M(A_k,B)$,
$C_k\to C\in \R^n$ for some $A\in S^{n-1}$ and $C\in \R^n$
with $\angle A_k B C_k \to 0$ as $k\to  \infty$ and $|BC_k|>3/4$.  Clearly therefore $C\in M(A,B)$, $|BC|\geq 3/4$ and
 $\angle ABC = 0$, hence $C\in (AB)$, which implies that $C$ is the midpoint of the segment $[AB]$ and hence
$|BC|=1/2$, this contradiction concluding the proof.
\end{proof}

\begin{remark}\label{rm_sc_mediatr1} The proof of Lemma~\ref{lm_scontr_mediatr1} depends essentially on the fact that
for $\{C\}:=M(A,B)\cap (A B)$ one has $C=C_0$, the midpoint of $[A B]$.
It is worth noting however that if the norm $\|\cdot\|$ is not assumed to be symmetric (i.e.\
one does not have $\|x\|=\|-x\|$ for all $x\in E$), then
$M(A,B)\cap (A B)$ is still a singleton $\{C\}$, but it does not in general coincide with $C_0$, and one only has
\[
c |AB| \leq |AC| \leq (1-c) |AB|
\]
for some $c\in (0,1)$ depending only on $\|\cdot\|$.
Thus the claim of Lemma~\ref{lm_scontr_mediatr1} should be changed in this case to
\[
|BC|\leq (1-\bar c) |AB|
\]
for some $\bar c\in (0,1)$ depending only on $\|\cdot\|$.
\end{remark}

\section*{Acknowledgements}
We thank Antoine Lemenant for introducing us the problem and checking our proof, and Joseph Gordon for careful reading of the manuscript and numerous useful comments.

\bibliographystyle{plain}
%\bibliography{mathopt}
\def\cprime{$'$} \def\cprime{$'$} \def\cprime{$'$} \def\cprime{$'$}

\end{document}